\documentclass[a4paper,11pt]{article}
\usepackage{indentfirst,latexsym,bm,amsmath,graphicx,amsthm,amssymb}
\usepackage{color}
\usepackage{framed}
\usepackage{relsize}

\newtheorem{theorem}{Theorem}
\newtheorem{proposition}{Proposition}

\newtheorem{lemma}{Lemma}

\newcommand*\samethanks[1][\value{footnote}]{\footnotemark[#1]}
\newcommand{\hl}{}

\newcommand{\tr}{^{\mathsf {\scriptscriptstyle T}}}
\newcommand{\str}{^{\mathsf {\scriptscriptstyle T}}}

\def\Expect{\mathbb{E}}
\def\Var{\mbox{Var}}
\def\Cov{\mbox{Cov}}
\def\Prob{\mathbb{P}}

\def\ind{\mathlarger{\mathbb{I}}}

\begin{document}
\title{Diffusion Parameters of Flows\\ in Stable Multi-class Queueing Networks}
\author{Sarat Babu Moka\thanks{School of Mathematics and Physics, The University of Queensland, Brisbane, Australia.}, Yoni Nazarathy\samethanks, 
Werner Scheinhardt\footnote{Department of Applied Mathematics, University of Twente, Enschede, The Netherlands.}
}

\date{\today}

\baselineskip=15pt
\maketitle
\begin{abstract}
We consider open multi-class queueing networks with general arrival processes, general processing time sequences and Bernoulli routing. The network is assumed to be operating under an arbitrary work-conserving scheduling policy that makes the system stable. 
We study the variability of flows within the network. Computable expressions for quantifying flow variability have  previously been discussed in the literature. However, in this paper, we shed more light on such analysis to justify the use of these expressions in the asymptotic analysis of network flows. Towards that end, we find a simple diffusion limit for the inter-class flows and establish the relation to asymptotic (co-)variance rates.
\end{abstract}

\section{Introduction}
\label{sec:intro}

The study of explicit performance measures of stable queueing networks has been at the heart of applied probability and operations research for the past half century. Initial results such as Burke's Theorem \cite{Burke0245}, indicating that the output of a stationary M/M/1 queue is a Poisson process have motivated the study of queueing output processes with the aim of using the output characteristics of one queue as the input of a downstream queue. While landmark results such as the product form solution of Jackson networks (c.f.\ \cite{jackson50jlq} or \cite{bookKelly1979}) have given much hope and practical utility, in the 1960's and 1970's it was well understood that explicit exact queueing network decomposition is in general not attainable. 

The lack of explicit solutions in general cases as well as the inability to exactly decouple most networks has motivated the study of queueing output processes as in \cite{bean1998output}, \cite{Daley0510}, \cite{bookDisneyKiessler1987}, \cite{DisneyKonig0283}, \cite{hautphenne2015intercept} and \cite{whitt2018heavy}. That line of work is also coupled with the development and study of heuristic queueing network decomposition schemes such as the Queueing Network Analyzer (QNA) \cite{whitt1983pqn} (see also \cite{kuehn1979approximate}), and many subsequent approximation methods (see for example the heuristics in \cite{kim2011modeling}). The typical approximating assumption made in such schemes is that each queue in isolation is a G/G/1 queue which can be analysed independently of the other queues. The input process is then approximated by taking into consideration both exogenous arrivals and departures from other queues in the network.

Some of the key ingredients needed for a network decomposition (such as QNA) are based on 
\[
\lambda_{\hl{k}} := \lim_{t \to \infty} \frac{\Expect [ \hl{A}_{\hl{k}}(t) ]}{t},
\quad
\mbox{and}
\quad
\sigma^2_{\hl{k}} = \lim_{t \to \infty} \frac{\Var \Big( \hl{A}_{\hl{k}}(t)  \Big)}{t},
\]
where $\hl{A}_{\hl{k}}(t)$ is the arrival counting process into the queue of class $\hl{k}$:
\[
\hl{A}_{\hl{k}}(t) := \hl{E}_{\hl{k}}(t) + \sum_{i} D_{i,{\hl{k}}}(t),
\]
with $\hl{E}_{\hl{k}}(t)$ representing the exogenous arrival counting process to that queue and $D_{i,j}(t)$ the number of items that have departed from queue $i$ and immediately arrived to queue $j$ during the time interval $[0,t]$. We refer to the counting process $D_{i,j}(t)$ as {\em flow} $i \to j$. The summation in $\hl{A}_{\hl{k}}(t)$ is over all flows $i \to {\hl{k}}$.

Finding $\lambda_{{\hl{k}}}$ exactly is typically a trivial matter based on the network routing matrix and exogenous arrival rates. As opposed to that, $\sigma^2_{{\hl{k}}}$ is more complex. In fact, computable expressions for $\sigma^2_{\hl{k}}$ have only been presented as part of the so-called ``innovations method'' in \cite{kim2011modeling}. Here, the author builds on earlier work \cite{KimRaviColm2005} by Kim, Muralidharan and O'Cinneide, and presents an expression that yields $\sigma^2_{\hl{k}}$ among other performance measures (see equation (42) in \cite{kim2011modeling}).  We independently arrived to similar formulas backed by rigorous proofs, for such performance measures. The performance measures that we cover include,
\begin{equation}
\label{eq:Sig1}
{\sigma}_{i,j} := \lim_{t \to \infty} \frac{\Cov \Big(\hl{A}_{i}(t),\hl{A}_{j}(t) \Big)}{t},
\end{equation}
and the asymptotic variability parameters of flows:
\begin{equation}
\label{eq:Sig2}
{\sigma}_{i \to j}^2  :=\lim_{t \to \infty} \frac{\Var \Big({D}_{i,j}(t) \Big)}{t}, \,\, \text{ and}
\quad
{\sigma}_{i_1\to j_1,i_2 \to j_2} := \lim_{t \to \infty} \frac{\Cov \Big({D}_{i_1,j_1}(t),{D}_{i_2,j_2}(t) \Big)}{t}.
\end{equation}

The focus and contribution of \cite{kim2011modeling} and \cite{KimRaviColm2005} is on advancing the state of the art in network decomposition approximations, and not on exact expressions for asymptotic variability nor on rigorous asymptotic analysis. Hence, in using such results 
one is left wondering about the meaning and rigour justifying validity of the expressions at hand. Specifically, there remain open questions regarding stability conditions, the usage in diffusion limits and the relationship to asymptotic variance rates. 

Our key contribution in the current paper is answering such questions as well as presenting detailed formulas for $\sigma^2_k$, $\sigma_{i,j}$, ${\sigma}_{i \to j}^2$ and ${\sigma}_{i_1\to j_1,i_2 \to j_2}$.
Our formulas 
hold for a wide class of stable networks. However, for concreteness we present our results in the context of stable multi-class queueing networks (of which generalized Jackson networks are a special case). Our main result,  Theorem~\ref{thm:main}, is formulated as a simple functional central limit theorem (FCLT) for the aforementioned processes and also ties the covariance structure of the limiting FCLT processes to system asymptotic variability parameters via uniform integrability. That is, expressions for  \eqref{eq:Sig1} and \eqref{eq:Sig2} and hence for $\sigma^2_{\hl{k}}$ are rigorously justified. 

In dealing with a stable queueing network, this could be viewed as a ``fundamental'' diffusion limit result similar to some of the results summarized in \cite{bookChenYao2001}. To the best of our knowledge within the context of diffusion limits, this fundamental result has been overlooked by previous authors with the exception of \cite{kim2011modeling} and \cite{KimRaviColm2005} that don't focus on rigorous asymptotic results. This is probably due to the fact that much of the exciting research in the field of diffusion approximations of queueing networks in the past three decades, has focused on critically loaded networks (c.f. \cite{BramsonDai0194}, \cite{dai1991steady}, \cite{reiman1984oqn} \cite{Williams0193},  as well as many other key references summarized in \cite{bookChenYao2001}, \cite{Glynn0518} \cite{bookMeyn2008}, and \cite{bookWhitt2001}). The seminal paper \cite{chen1991stochastic}, does consider diffusion approximations for queueing networks in all regimes (under-loaded, balanced and over-loaded), yet the inter-queue flows are not considered in that paper. Also, in Section~4 of the early influential paper \cite{harrison1988brownian}, Harrison considers flows, but the analysis there is only for the uninterrupted (primitive) processes and not for the true flows $\hl{A}(\cdot)$ and $D(\cdot)$ as we have here.

As described in our main diffusion result, the asymptotic variability of flows is driven by two components: (i) The variability of the arrival flows; (ii) The variability resulting from the Bernoulli routing. In stable networks, the variability of queue sizes (related also to service time distributions) does not play a role. Since asymptotic variability of flows only depends on the interplay of the arrival process variability and the Bernoulli routing, we are also motivated to present an alternative way for quantifying the asymptotic variability parameters: {\em networks with zero service times}. In such networks, jobs that arrive to the network traverse instantaneously through the classes/queues until they depart, and hence the total count of jobs passing on flow $i\to j$ is
\[
\sum_{k} \sum_{\ell = 1}^{\hl{E}_k(t)} N_{i,j|k}(\ell),
\]
where the outer sum is over all classes and $N_{i,j|k}(\ell)$ are counts of the number of passes on $i\to j$ for the $\ell$'th job arriving exogenously to $k$. Using elementary calculations, we find the asymptotic variances and covariances of such processes, and prove they are the same as those originating from the diffusion parameters. It is the zero service time view which allows us to establish uniform integrability and to relate the diffusion parameters to asymptotic variability parameters.

The structure of the sequel is as follows. In Section~\ref{sec:model} we summarize our results in a main theorem together with the notation and assumptions of the model. Then the three sections that follow constitute the proof. In Section~\ref{sec:difflimit} we  present the calculation of the diffusion parameters and diffusion limit. In Section~\ref{sec:zerTime} we present the alternative view of the network based on zero service times. In Section~\ref{sec:asympVarAndUI} we relate the diffusion parameters to asymptotic variance and establish the required UI. We then follow with Section~\ref{sec:example} where we present a numerical example, and compare with the innovations method of \cite{kim2011modeling}. For this comparison and for the convenience of readers, we spell out the details of the relevant results from \cite{kim2011modeling}. Readers are encouraged to read Section~\ref{sec:example} in conjunction with Section~\ref{sec:model}. Closing remarks are in Section~\ref{sec:coclusion}.


\section{Model and Main Result}
\label{sec:model}

We consider open multi-class queueing networks (MCQN) operating under an arbitrary resource allocation policies subject to Bernoulli routing. Special cases include generalized Jackson queueing networks as described in \cite{bookChenYao2001} as well as many other examples appearing in \cite{bramsonBook2008}. Consider a queueing network of $\hl{J}$ servers serving $K$ classes of customers denoted by $1,\ldots,K$. For each class there is a unique server $s(k)$ and let $C_j = \{k : s(k) = j\}$ denote the constituency of server $j$, that is, $C_j$ is the set of all the classes served at server $j$. 

The evolution of the network is driven by the following primitive sequences of random variables:  $\{\xi_k({\hl{\ell}})\}_{{\hl{\ell}}=0}^\infty$  is the sequence of exogenous inter-arrival times to class $k$ and $\{\eta_k({\hl{\ell}})\}_{{\hl{\ell}}=0}^\infty$ is the sequence of service times in class $k$. The sequence \hl{$\{\phi_{i,j}({\hl{\ell}})\}_{{\hl{\ell}}=1}^\infty$} of indicator variables determine if the ${\hl{\ell}}$'th job departing from class ${\hl{i}}$ moves to class $\hl{j}$ (this is indicated via \hl{$\phi_{i,j}(\ell) = 1$}) where $\hl{i} \in \{1,\ldots,K\}$ and $\hl{j} \in \{0,\ldots,K\}$ with $\hl{j} = 0$ implying the job departing from the network. \hl{For any $\ell$,}
\begin{equation}
\label{eq:exactly1destination}
\hl{
\sum_{j = 0}^K \phi_{i,j}({\hl{\ell}}) = 1.
}
\end{equation}

These primitives are assumed to exist on a joint probability space and construct the primitive processes $\big(\hl{E}, S, \Phi \big)$, as we describe now.
\[
\hl{E}_k(t) = \max \{ n \ge 0 ~:~ \sum_{{\hl{\ell}}=0}^{n-1}\xi_k({\hl{\ell}}) \le t \}
\qquad \text{and} \qquad
S_k(t) = \max \{ n \ge 0 ~:~ \sum_{{\hl{\ell}}=0}^{n-1}\eta_k({\hl{\ell}}) \le t \},
\]
with the convention that summation from $0$ to $-1$ is $0$. The counting processes, $\hl{E}_k(t)$ and $S_k(t)$ represent the number of exogenous arrivals to class $k$ during $[0,t]$ and the number of jobs served during uninterrupted service in class $k$ during $[0,t]$, respectively. Further for ${\hl{\ell}}=1,2,\ldots$, let,
\[
\hl{
\Phi_{i,j}({\hl{\ell}}) = \sum_{\ell'=1}^{\hl{\ell}} \phi_{i,j}(\ell'),
}
\]
which denotes the number of items routed from class $\hl{i}$ to class $\hl{j}$ out of the first $\hl{\ell}$ items served at~$\hl{i}$. Due to \eqref{eq:exactly1destination}, \hl{for any $\ell$,}
\begin{equation}
\label{eq:phiTotal}
\hl{
\sum_{j=0}^K \Phi_{i,j}(\ell)= \ell,
\quad
i=1,\ldots,K.
}
\end{equation}


 The primitive processes $(\hl{E}, S, \Phi)$ interact to yield the network processes $(T,Q,D,\hl{A})$ which we define now. 
Let $T_k(t)$ denote the work (in units of time) allocated towards serving class $k$ during the time interval $[0,t]$. In general, $T$ is policy dependent as it captures how server effort is allocated among classes. We have for all $j \in \{1,\ldots,\hl{J}\}$,
\[
\hl{\sum_{k \in C_j} T_k(t) \le t.}
\]
With $T_k(t)$ at hand, the actual number of class $k$ jobs served during $[0,t]$ is $S_k \big( T_k(t) \big)$. Further, composing with $\Phi$ we define the {\em inter-class flows} via,
\begin{equation}
\label{eq:depComposition}
D_{i,j}(t) = \Phi_{i,j}\Big(S_i \big(T_i(t) \big)\Big),
\quad i=1,\ldots,K,\,\,  j=0,\ldots,K.
\end{equation}
Let $Q_k(t)$ denote the number of items \hl{of class $k$ at time $t$ in the system (queue or in service)}. \hl{We refer to this number as the queue length, and it satisfies} 
\hl{
\[
Q_k(t) = \hl{A}_k(t) - \sum_{j=0}^K D_{k,j}(t) + Q_k(0),
\]
}
where the (total) arrival process to class $k$ is,
\begin{equation}
\label{eq:queueArrivals}
\hl{A}_k(t) = \hl{E}_k(t) + \sum_{i=1}^K D_{i,k}(t).
\end{equation}

\hl{
In our exposition we assume $Q_k(0) = 0$ and thus the queue length equation becomes
\begin{equation}
\label{eq:queueDynamics}
Q_k(t) = \hl{A}_k(t) - \sum_{j=0}^K D_{k,j}(t).
\end{equation}
Our results below can be generalized for cases where $Q_k(0)$ is at some fixed positive quantity or is random. For clarity of the exposition we omit these details.
}

In the treatment below, the vectors $Q, T, \hl{E, A}$ and $S$ (and their ``bar", ``hat" and ``tilde" versions as defined below) are treated as $K$-dimensional column vectors. Further, let $\Phi$ and $D$ be $K^2$ dimensional column vectors with the elements ordered in lexicographic order with the elements $D_{k,0}$ omitted. For example,
\[
D = \Big( D_{1,1}, \ldots,D_{1,K},D_{2,1},\ldots,D_{2,K},\ldots \ldots \ldots \ldots,D_{K,1},\ldots,D_{K,K} \Big)\tr.
\]

\subsection*{Probabilistic Assumptions}

Throughout the paper, we make use of the following assumptions on the network primitives. Without loss of generality assume for some $1 \leq L \leq K$ that only the first $L$ classes have non-null exogenous arrivals, and for the other $K-L$ classes, the exogenous inter-arrival times are infinite (no arrivals). These are assumptions (A1)--(A4). Note that (A3) is not needed for our main result, but is needed to satisfy positive Harris recurrence (stability) results in general.
\begin{enumerate}
 \item [(A1)] $\{\xi_k({\hl{\ell}})\}_{{\hl{\ell}}=0}^\infty$  are i.i.d. sequences and mutually independent over all $k=1,\ldots,L$. Furthermore, independent of inter-arrival times, the sequences $\{\eta_k({\hl{\ell}})\}_{{\hl{\ell}}=0}^\infty$  are i.i.d. sequences and mutually independent over all $k=1,\ldots,K$.
%
 \item[\hl{(A2)}]  $0< \mathbb{E}[(\xi_k(\hl{0}))^{\hl{3+\varepsilon}}] < \infty$ for $k = 1,\dots,L$ and
	     $0< \mathbb{E}[(\eta_k(\hl{0}))^{\hl{3+\varepsilon}}]< \infty$ for $k = 1,\dots, K$ \hl{and $\varepsilon > 0$}.
\item[\hl{(A3)}] Independent of inter-arrival times and service times, each vector $(\phi_{k,0}(\hl{\ell}),\ldots,\phi_{k,K}(\hl{\ell}))\tr$ for $k=1,\ldots,K$ and $\hl{\ell\ge 1}$ follows a multinomial distribution with a single success and probability vector $(p_{k,0},p_{k,1},\ldots,p_{k,K})\tr$ with $p_{k,j} \ge 0$, and  \hl{$p_{k,0} = (1- \sum_{j=1}^K p_{k,j})\ge 0$}. We denote by $P$ the $K \times K$ matrix of $p_{k,j},\, k,j=1,\ldots,K$. 
\end{enumerate}
Assumption (A1) is standard. Assumption (A2) \hl{yields finite moments and} is used for diffusion limits and uniform integrability. 
For more explanation on this assumptions, see  \cite{bramsonBook2008}. Assumption \hl{(A3)} is the standard ``Bernoulli routing'' assumption implying that each $K+1$ dimensional vector $(\phi_{k,0}(\hl{\ell}),\ldots,\phi_{k,K}(\hl{\ell}))\tr$  has a single entry that is $1$ and $K$ zero entries.

Since $\xi$ and $\eta$ have finite second moments, they also have finite first moments and we denote,
\[
\hl{\alpha_k= \frac{1}{\mathbb{E}[\xi_k(0)]},}
\qquad
\mbox{and}
\qquad
\hl{\mu_k= \frac{1}{\mathbb{E}[\eta_k(0)]}.}
\]
We also denote by $\alpha$ the vector of $\alpha_{\hl{k}}$ and $\mu$ the vector of $\mu_{\hl{k}}$.

\subsection*{Structural Network Assumptions}

We assume that the network is {\em open} and {\em stabilizable} via the following two assumptions:

\begin{enumerate}
\item[\hl{(A4)}]  The matrix $P$ has a spectral radius less than $1$ that is, $I-P\tr$ is non-singular.
\end{enumerate}
We now denote $\lambda = (I-P\tr)^{-1} \alpha$ and $\lambda_{i,j}:=\lambda_i \, p_{i,j}$. Now assume
\begin{enumerate}
\item[\hl{(A5)}] $\sum_{k \in {\cal C}_{\hl{j}}} \frac{\lambda_k}{\mu_k} < 1$ for every server ${\hl{j}}$.
\end{enumerate}

\noindent
An additional assumption that we make is that policies are {\em work conserving}. For this we denote the idle time process of server $j$ via,
\[
\hl{{\cal I}_j(t)} = t - \sum_{k \in C_j} T_k(t).
\]
Now the work conserving assumption is 
\begin{enumerate}
\item[\hl{(A6)}] \hl{$\displaystyle \int_0^t \Big(\sum_{k \in C_j} Q_k(u)\Big) \,\hl{d\, {\cal I}_j(u)} = 0$} for all $t \ge 0$ and all servers $j$.
\end{enumerate}

Assumption \hl{(A4)} means that the network is open. Assumption \hl{(A5)} is used for stability, a concept that we discuss in further detail below. The assumption implies that there is enough capacity in the network. If it is violated, then it is easy to show that the network cannot be stabilized, \cite{bramsonBook2008}. As opposed to that under \hl{(A5)}, much research has gone into finding scheduling policies that stabilize the network. For general MCQN such policies exist, see for example \cite{bramson2016proportional} \hl{or \cite{Weiss21}} and references therein. In the case of generalized Jackson networks (single class meaning that $|C_j| = 1$ for all servers $j$), under \hl{(A5)} and \hl{(A6)} networks are stable, see for example \cite{baccelli1994ergodicity}.

\subsection*{Scaling Limits}
For $n=1,2,\ldots$ and a function $U(t)$,
denote $\overline{U}^n(t) = U( nt )/n$. We say that a fluid limit of $U$ exists if $\lim_{n \to \infty} \overline{U}^n(t) = \overline{U}(t)$ exists uniformly on compact sets (u.o.c) almost surely. Further, when the limit $\overline{U}(t)$ exists, denote,
\begin{equation}
\label{eq:defDifScale}
\widehat{U}^n(t) = \frac{U(nt) -  \overline{U}(n t)}{\sqrt{n}},
\quad
n=1,2,\ldots.
\end{equation}
In cases where the above sequence converges weakly on Skorohod $J_1$ topology to a limiting process, $\widehat{U}(t)$, we denote,
\[
\widehat{U}^n \Rightarrow \widehat{U}.
\]
For discrete time processes replace $U(nt)$ by $U(\lfloor nt \rfloor)$. See \cite{bookChenYao2001}, Chapter~5 for brief background of weak convergence in the context of queueing networks. An extensive treatment is in \cite{bookWhitt2001}.

As a consequence of the assumptions (A1) and (A2) (for first moments), the primitive processes satisfy a functional strong law of large numbers (FSLL) yielding fluid limits $\hl{\overline{E}}_{\hl{k}}(t) = \alpha_{\hl{k}} t$ and $\overline{S}_{\hl{k}}(t) = \mu_{\hl{k}} t$. Further, from \hl{(A3)}, $\overline{\Phi}_{i,j}(\ell) = p_{i,j} \ell$. 

As a consequence of assumptions \hl{(A1) and (A2)} the primitive processes satisfy functional central limit theorems (FCLT). Specifically $\hl{\widehat{E}}_k(t)$ are Brownian motions with diffusion coefficients \hl{(also sometimes known as volatility coefficients)},
\begin{equation}
\label{eq:scvFCLTarrivals}
v_k= \alpha_k c_k^2,
\qquad
\mbox{where}
\qquad
c_k^2 =\frac{\mathbb{E}[\xi_k(\hl{0})^2]}{  \mathbb{E}[\xi_k(\hl{0})]^2} - 1.
\end{equation}
Similar diffusion limits exists for the service processes however these do not play a role in our limiting results. The routing processes also have diffusion limits due to \hl{(A3)}. We have that, 
\begin{equation}
\label{eq:PhiFCLTStuff}
\widehat{\Phi}_{k,\cdot} (t) = \Big(\widehat{\Phi}_{k,1}(t),\ldots, \widehat{\Phi}_{k,K}(t) \Big)
,
\quad
 k=1,\ldots,K,
\end{equation}
 are $K$-dimensional Brownian motions with covariance matrices $\Gamma_k$, having the $i,j$'th entry $p_{k,i}(\delta_{i,j}-p_{k,j})$, where $\delta_{i,j}$ is the Kronecker delta. See \cite{chen1991stochastic} or \cite{bookWhitt2001}.
All these results are for primitive processes, our theorem deals with scaling limits of the flows.



\subsection*{Stability Using Fluid Models}

We now briefly give an overview of the fluid stability framework, deferring details to the literature for the sake of brevity. For a MCQN and a scheduling policy, we can associate a set of deterministic equations called the {\em fluid model (equations)}. Such equations are spelled out in detail in \cite{bramsonBook2008} page 104, equations (4.50) - (4.55). The fluid model description in \cite{bramsonBook2008} summarizes key ideas from \cite{chen1995faa}, \cite{Dai108},  \cite{DaiMeyn0337} and others. In general, each scheduling policy may induce a different set of equations. Examples are in \cite{bramsonBook2008}. A key object in the fluid model equations is the queue fluid limit, $\{Z(t)\}_{t \ge 0}$, a $K$ dimensional vector (using the notation of \cite{bramsonBook2008}).

The concept of {\em fluid model stability} then requires that there exist some finite time $t^*$ such that if $\sum_{k=1}^K Z_k(0) = 1$ then $\sum_{k=1}^K Z_k(t) = 0$ for $t \ge t^*$. Much of the literature on MCQN has dealt with proving fluid model stability associated with different networks and scheduling policies, see \cite{bramsonBook2008}. A key result in \cite{Dai108} (also summarized in \cite{bramsonBook2008}) connects fluid model stability to positive Harris recurrence of the associated Markov process describing the MCQN. In general the most accepted stability notion of a stochastic MCQN is positive Harris recurrence. Of notable mention are generalized Jackson networks (single class) which are stable under any work conserving policy and assumptions \hl{(A4)--(A5)}. Our main theorem requires the fluid model of the network to be stable.

\hl{Assumptions (A1)--(A3) and an additional technical assumption requiring the inter-arrival times to be unbounded and spread-out (see for example (1.4) and (1.5) in \cite{Dai108}) can be used to show that a stable fluid model yields positive harris recurrence of the network. In this paper, we don't explicitly require positive harris recurrence, and hence such an additional technical assumption is not needed. }

A related stability notion that we use in the sequel is weak stability. The fluid model is {\em weakly stable} if when $Z(0) = 0$ then $Z(t) = 0$ for all $t \ge 0$.  Weak stability clearly follows from fluid model stability since our networks are time-homogenous.

\subsection*{Main Result}

We now set up some matrices and vectors used in our main theorem. Use $\mathbf{1}$ to denote the $K$-dimensional vector of ones and define the $K \times K^2$ matrix $B :=  \mathbf{1}\tr \otimes I$ where $\otimes$ is the Kronecker product \hl{and here $I$ is the $K \times K$ identity matrix}. Further denote the $K^2 \times K$ matrix, 
\[
P_c :=
\begin{bmatrix}
P\tr \, e_{1,1}  \\
P\tr\, e_{2,2}  \\
\vdots \\
P\tr\, e_{K,K}
\end{bmatrix},
\]
where $e_{i,j}$ is a $K \times K$ matrix with all entries $0$ except for the $i,j$'th entry being $1$. Now define the $K\times (K + K^2)$ matrix $G$ and the $K^2 \times (K+K^2)$ matrix $H$, respectively, as,
\begin{align}
G &:= \begin{bmatrix}
    (I-P\tr)^{-1} &      (I-P\tr)^{-1}B
    \end{bmatrix}, \label{eqn:G_def}\\
H &:=
\begin{bmatrix}
P_c(I-P\tr)^{-1}
&
I_{K^2} + P_c (I-P\tr)^{-1}B
\end{bmatrix}\label{eqn:H_def}.
\end{align}
Also define the $(K+K^2) \times (K+K^2)$ covariance matrix for the exogenous arrival processes and the routing processes,
\begin{align}
\Sigma^{(P)}
:=
\begin{bmatrix}
\mbox{diag}(v_k^2) & & & 0\\
&\lambda_1 \Gamma_1 &          &                                              \\
 &        & \ddots   &                                                  \\
0  &       &          & \lambda_K \Gamma_K                               \\
\end{bmatrix},
\label{eqn:Sigma_P}
\end{align}
where $\mbox{diag}(v_k^2)$ is a diagonal matrix with elements $v_k^2$. Further, for any $i,j \in \{1,\ldots,K\}$ define the $K$ dimensional vector $m(i,j)$ as follows:
\begin{equation}
\label{eq:MijFirst}
m(i,j) := (I-P)^{-1} e_{i,i} P_{\cdot,j},
\end{equation}
where $P_{\cdot,j}$ is the $j$'th column of $P$. As further elaborated on in Section~\ref{sec:zerTime}, the $k$'th entry of the column vector $m(i,j)$ is the expected number of transitions from state $i$ to state $j$ in a Markov chain whose transient component is specified by $P$ and initial state is set to $k$.

We now present our main result. Relationships to the results of \cite{kim2011modeling} are in Section~\ref{sec:example}. 

\begin{theorem}
\label{thm:main}

Consider a multi-class queueing network and assume \hl{(A1)--(A6)} hold. If the fluid model of the network (incorporating the scheduling policy) is stable, then

(i) The sequences $\hl{\widehat{A}}^n$ and $\widehat{D}^n$ converge weakly to drift-less Brownian motion processes with covariance matrices,
\begin{equation}
\label{eq:covExpression}
\Sigma^{(\hl{A})} :=
G\,
\Sigma^{(P)}
\,
G\tr,
\quad
\mbox{and}
\quad
\Sigma^{(D)} := H
\,
\Sigma^{(P)}
\,
H\tr,
\end{equation}
respectively.

\vspace{10pt}

(ii) The asymptotic variability parameters, as defined in \eqref{eq:Sig1} and \eqref{eq:Sig2}, can be read off from the diffusion parameters. Namely,
\[
\sigma_{i_1 \to j_1, i_2 \to j_2} = \Sigma^{(D)}_{(i_1-1)K+j_1,~(i_2-1)K + j_2},
\qquad
\sigma_{i,j} = \Sigma^{(\hl{A})}_{i,j}.
\]

\vspace{10pt}

(iii) An alternative calculation for the asymptotic variability parameters is,
\begin{eqnarray}
\label{eq:mainThmiii-1}
\sigma_{i_1 \to j_1, i_2 \to j_2} &=& 
m_{j_1}(i_2,j_2)  \alpha\str m(i_1,j_1)  + m_{j_2}(i_1,j_1) \alpha\str  m(i_2,j_2)\\
\nonumber
&&+ (v^2 - \alpha)\str \big(m(i_1,j_1) \bullet m(i_2,j_2) \big) , \\
\label{eq:mainThmiii-1_equal}
\sigma_{i \to j}^2 &=& 
(1+2 m_{j}(i,j))  \alpha\str m(i,j)  + (v^2 - \alpha)\str \big(m(i,j) \bullet m(i,j) \big) , \\
\label{eq:mainThmiii-2}
\sigma_{i,j} &=&v_{i}^2 \sum_{k=1}^K m_{i}(k,j) + v_{j}^2 \sum_{k=1}^K m_{j}(k,i)
+ \sum_{k_1=1}^K \sum_{k_2=1}^K \sigma_{k_1 \to i, k_2 \to j}
\end{eqnarray}
where $m_k(i,j)$ is the $k$-th entry of the vector $m(i,j)$ and $(x \bullet y)$ signifies the vector resulting from element-wise product of the vectors $x$ and $y$.
\end{theorem}

Since the result may appear quite technical, we demonstrate the applicability on a specific network example in Section~\ref{sec:example}. 

The remainder of the paper establishes the proof. For this we now put forward foundations utilizing several results from the literature. Specifically equations \eqref{eq:barQis0}, \eqref{eq:barTdoesTheJob}, \eqref{eq:momentCondMeynDai}, and \eqref{eq:AUIass}, that we obtain now, are used in the proofs.

 Using weak stability, and assumptions \hl{(A4)--(A6)}, Theorem~4.1 in \cite{chen1995faa} ensures $\overline{Q}(t)$ and $\overline{T}(t)$ exist and for any class~$k$, 
\begin{equation}
\label{eq:barQis0}
\overline{Q}_k(t) = 0,
\end{equation}
and
\begin{equation}
\label{eq:barTdoesTheJob}
\overline{T}_k(t) = \frac{\lambda_k}{\mu_k} t.
\end{equation}
Using fluid model stability and assumptions \hl{(A1) and (A2)}, Theorem~4.1 (ii) in \cite{DaiMeyn0337} states that for every initial state $x$ of the associated Markov process of the network (and policy),
 \begin{equation}
 \label{eq:momentCondMeynDai}
 \lim_{t \to \infty} \mathbb{E}_x \left[Q_k(t)^{\hl{2+\varepsilon}} \right]  
 \leq c,
 \end{equation}
 for some constant $c$, any class $k$, \hl{and some $\epsilon$ as in (A2)}.

We now refer to equation (5.18) on page 60 of \cite{gut2009stopped}. We have that under assumptions (A1) and (A2) for the arrival process, for each class $k$ and some $t_0>0$,
\begin{equation}
\label{eq:AUIass}
\left\{ 
\frac{(\hl{E}_k(t) - \alpha_k t)^2}{t}, \, t \ge t_0
\right\}
\quad
\mbox{is Uniformly Integrable}.
\end{equation}

The proof is structured as follows:
(i) is established in Section~\ref{sec:difflimit}. (iii) is established in Section~\ref{sec:zerTime}. (ii) relies on the development of (iii) and is established in Section~\ref{sec:asympVarAndUI}.


\section{The Diffusion Parameters}
\label{sec:difflimit}

Using the scaling definition \eqref{eq:defDifScale}, equations \eqref{eq:phiTotal},  \eqref{eq:queueDynamics} and \eqref{eq:queueArrivals} are easily manipulated for all \hl{$k=1,\ldots,K$ and $n \ge 1$} to yield,
\begin{align}
\label{eq:Ddiff}
0 &= \hl{\sum_{j=0}^K \widehat{\Phi}^n_{k,j}(\ell)},
\quad
\hl{\ell}=1,2,\ldots,\\
\label{eq:currentLaw}
\widehat{Q}^n_k(t) &= \hl{\widehat{E}}^n_k(t) + \sum_{\hl{i}=1}^K \widehat{D}^n_{\hl{i},k}(t)
- \sum_{j=0}^K \widehat{D}^n_{k,j}(t),
\quad
t\ge 0,\\
\label{eq:Ediff}
\hl{\widehat{A}}^n_k(t) &= \hl{\widehat{E}}^n_k(t) + \sum_{i=1}^K \widehat{D}_{i,k}^n(t),
\quad
t\ge 0.
\end{align}

Observe that, the property $\overline{T}_k(t) = \frac{\lambda_k}{\mu_k} t$ implies,
\begin{equation}
\label{eq:DnLim}
\lim_{n \to \infty} \overline{D}_{i,j}^n(t) :=\overline{D}_{i,j}(t)=\overline{\Phi}_{i,j}( \overline{S}_i( \overline{T}_i(t)))=p_{i,j} \lambda_i t,\,\,\,\mbox{u.o.c.}.
\end{equation}

Lemmas~\ref{lem:1}--\ref{lem:4} below summarize straight forward algebraic manipulations of these equations. Then this leads to a simple diffusion limit that follows from Donsker's theorem (see \cite{bookChenYao2001}, Chapters~5-7 or  \cite{Glynn0518} for background). Techniques similar to those employed here are also in \cite{nazarathy2010positive}, applied to queueing networks that generate their own input.
The basic idea is to represent the diffusion scaled processes, $\widehat{D}^n$ and $\widehat{T}^n$ in-terms of the following ``tilde" processes,
\[
\widetilde{\Phi}_{i,j}^n(t) := \widehat{\Phi}_{i,j}^n\Big(\overline{S}_i^n \big(\overline{T}_i^n(t)\big)\Big),
\quad
\mbox{and}
\quad
 \widetilde{S}_k^n(t) := \widehat{S}_k^n(\overline{T}_k^n(t)),
\]
which in-turn have diffusion limits based on the primitive processes.
%
\begin{lemma}
\label{lem:1}
For $i=1,\ldots,K$ and $j=0,\ldots,K$,
\begin{eqnarray}
\label{eq:depDecompose}
\widehat{D}^n_{i,j}(t) &=& \widetilde{\Phi}^n_{i,j}( t )
+ p_{i,j} \widetilde{S}^n_i(t) +
p_{i,j} \mu_i \widehat{T}^n_i(t).
\end{eqnarray}
\end{lemma}
\proof{}
Use $D_{i,j}(nt) = \Phi_{i,j}(S_i(T_i(nt))) = \Phi_{i,j}(n \overline{S}_i^n( \overline{T}_i^n(t)))$ and \eqref{eq:DnLim} to get,
\begin{eqnarray*}
\widehat{D}^n_{i,j}(t)
&=& \frac{\Phi_{i,j}(n \overline{S}_i^n( \overline{T_i}^n(t)))- p_{i,j} \lambda_i n t}{\sqrt{n}} \\
&=& \frac{\Phi_{i,j}(n \overline{S}_i^n( \overline{T_i}^n(t)))-  p_{i,j} n\overline{S}_i^n(\overline{T}_i^n(t))}{\sqrt{n}}
+ \frac{  p_{i,j} n \overline{S}_i^n(\overline{T}_i^n(t)) -  p_{i,j} \mu_i n \overline{T}_i^n(t)}{\sqrt{n}} \\
&& + \frac{   p_{i,j} \mu_i n \overline{T}_i^n(t) - p_{i,j} \lambda_i n t}{\sqrt{n}}.
\end{eqnarray*}
Now, (\ref{eq:depDecompose}) follows.

\endproof

Denote by $M$ the diagonal matrix with diagonal elements $\mu_k^{-1}$. We now have,

\begin{lemma}
\label{lem:2}
The diffusion scaled time allocation can be written as:
\hl{\begin{equation}
\label{eq:t-hat}
\widehat{T}^n(t) =
M (I-P\tr)^{-1}
\Big(
\hl{\widehat{E}}^n(t) +  B \widetilde{\Phi}^{n}(t)  -
   \widehat{Q}^n(t)
\Big) - M \widetilde{S}^n(t).
\end{equation}}
\end{lemma}
\proof{}
Substituting (\ref{eq:depDecompose}) into (\ref{eq:currentLaw}) we have:
\begin{eqnarray*}
\widehat{Q}^n_k(t) &=& \hl{\widehat{E}}^n_k(t) + \sum_{\hl{i}=1}^K
\left(\widetilde{\Phi}^n_{\hl{i},k}(t)
+ p_{\hl{i},k} \widetilde{S}^n_{\hl{i}}(t) +
p_{\hl{i},k} \mu_{\hl{i}} \widehat{T}_{\hl{i}}^n(t)\right)  \\
&& - \sum_{j=0}^K
\left(\widetilde{\Phi}^n_{k,j}(t)
+ p_{k,j} \widetilde{S}^n_k(t) +
p_{k,j} \mu_k \widehat{T}^n_k(t)\right) \\
&=&
\hl{\widehat{E}}^n_k(t) + \sum_{\hl{i}=1}^K
\left(\widetilde{\Phi}^n_{\hl{i},k}(t)
+ p_{\hl{i},k} \widetilde{S}^n_{\hl{i}}(t) +
p_{\hl{i},k} \mu_{\hl{i}} \widehat{T}_{\hl{i}}^n(t)\right) - \widetilde{S}^n_k(t) - \mu_k \widehat{T}^n_k(t) \\
&=&
\hl{\widehat{E}}_k^n(t) + \sum_{{\hl{i}}=1}^K \widetilde{\Phi}_{{\hl{i}},k}^n(t) -
 \big( \widetilde{S}_k^n(t)  - \sum_{{\hl{i}}=1}^K p_{{\hl{i}},k} \widetilde{S}_{\hl{i}}^n(t)  \big)
- \big( \mu_k \widehat{T}_k^n(t)  - \sum_{{\hl{i}}=1}^K p_{{\hl{i}},k} \mu_k \widehat{T}_{\hl{i}}^n(t)  \big),
\end{eqnarray*}
where in the second step we used (\ref{eq:Ddiff}) and $\sum_{j=0}^K p_{i,j} =1$. In vector/matrix form this reads:
\begin{eqnarray*}
\widehat{Q}^n(t) &=&
\hl{\widehat{E}}^n(t) +
B {\widetilde{\Phi}^n(t)}
- (I -  P\tr) \widetilde{S}^n(t) - (I -  P\tr) M^{-1}  \widehat{T}^n(t).
\end{eqnarray*}
Now (\ref{eq:t-hat}) follows by multiplying both sides by $M (I-P\tr)^{-1}$.

\endproof
We now have,

\begin{lemma}
\label{lem:3}
\[
\widehat{D}^n(t) =
\begin{bmatrix}
H & 0_{K \times K} 
\end{bmatrix}
\begin{bmatrix}
\hl{\widehat{E}}^n(t) \\
\widetilde{\Phi}^n(t) \\
\widehat{S}^n(t) 
\end{bmatrix}
-P_c\, (I-P\tr)^{-1} \widehat{Q}^n(t).
\]
\end{lemma}
\proof{}
Equations (\ref{eq:depDecompose}) are,
\[
\widehat{D}^n(t) = \widetilde{\Phi}^n(t) + P_c \, \Big( \widetilde{S}^n(t) + M^{-1} \widehat{T}^n(t) \Big).
\]
Substituting (\ref{eq:t-hat}) in the above, $\widetilde{S}^n(t)$ drops out of the equation, and we obtain,
\[
\widehat{D}^n(t) = \Big(I_{K^2} + P_c \,(I-P\tr)^{-1}B \Big)\widetilde{\Phi}^n(t) + P_c\,(I-P\tr)^{-1} \hl{\widehat{E}}^n(t) - P_c(I-P\tr)^{-1} \widehat{Q}^n(t). 
\]
\endproof

Observe from Lemma~\ref{lem:3} that $\widehat{D}^n$ depends on $\widehat{S}^n$ only through $\widehat{Q}^n$. We may now represent the analogous result for $\hl{\widehat{A}}^n$, this time omitting the primitive sequence $\widehat{S}^n$ from the representation.

\begin{lemma}
\label{lem:4}
\[
\hl{\widehat{A}}^n(t) =
G
\left[
\begin{array}{c}
\hl{\widehat{E}}^n(t) \\
\widetilde{\Phi}^n(t) 
\end{array}
\right]
- B\, P_c(I-P\tr)^{-1} \widehat{Q}^n(t),
\]
\end{lemma}
\proof{}
We use \eqref{eq:Ediff} and the previous lemma:
\begin{eqnarray*}
\hl{\widehat{A}}^n(t) &=& B \,\widehat{D}^n(t) + \hl{\widehat{E}}^n(t) \\
&=& B\, H 
\left[
\begin{array}{c}
\hl{\widehat{E}}^n(t) \\
\widetilde{\Phi}^n(t)  
\end{array}
\right]
- 
B\,  P_c(I-P\tr)^{-1} \widehat{Q}^n(t)
+
[ I_K ~~ 0 ] 
\left[
\begin{array}{c}
\hl{\widehat{E}}^n(t) \\
\widetilde{\Phi}^n(t)  
\end{array}
\right]\\
&=& \Big( B\, H + [I_K ~~ 0 ] \Big) 
\left[
\begin{array}{c}
\hl{\widehat{E}}^n(t) \\
\widetilde{\Phi}^n(t) 
\end{array}
\right]
- B\, P_c(I-P\tr)^{-1} \widehat{Q}^n(t).
\end{eqnarray*}
Since $B\, P_c = P\tr$ and $P\tr\, (I - P\tr)^{-1} = (I - P\tr)^{-1} - I$,
\begin{align*}
\Big( B\, H + [I_K ~~ 0 ] \Big) &= \begin{bmatrix} B\, P_c(I - P\tr)^{-1} + I & \left(I + B\, P_c (I - P\tr)^{-1}\right)\, B\end{bmatrix} = G. 
\end{align*}
\endproof
\noindent
We can now establish the diffusion limit in our main theorem.
\medskip

\noindent
{\em Proof of Theorem~\ref{thm:main} (i):}
%

\hl{
We first establish the process level convergence $\widehat{Q}^n \Rightarrow 0$ as $n \to \infty$. We rely on Lemma~3.19 in \cite{JacodShiryaev2003} stating that if two c\`{a}dl\`{a}g processes share the same finite dimensional distributions then their probability laws are identical. The $0$ process is clearly c\`{a}dl\`{a}g. We now show that the limiting process of $\widehat{Q}^n$ is also c\`{a}dl\`{a}g and converges on finite dimensional distributions to the tuple of zero vectors.  
}

\hl{
Using the union bound and the Markov inequality, for any $\tilde{\varepsilon} > 0$
\begin{align}
\label{eqn:Markov_inq}
\Prob\left(|\widehat{Q}^{n}(t)| > \tilde{\varepsilon} \right)  \le 
\sum_{k=1}^K \Prob\left(\frac{Q_{k}(nt)}{\sqrt{n}} > \frac{\tilde{\varepsilon}}{K} \right) \leq 
\frac{K^{2 + \varepsilon}}{{n^{1 + \varepsilon/2} \, {\tilde{\varepsilon}^{2+\varepsilon} }}}\sum_{k=1}^K  \Expect\left[Q_{k}(nt)^{2 + \varepsilon} \right].
\end{align}
Now using the finite moment convergence result \eqref{eq:momentCondMeynDai} and the Borel-Cantelli Lemma, since the series summing the probabilities of the left hand side of \eqref{eqn:Markov_inq} converges, 
it holds that $|\widehat{Q}^{n}(t)| \to 0$ as $n \to \infty$ almost surely. Hence for class $k$ and every fixed $t$
$\lim_{n \to \infty} \widehat{Q}^n_k(t) = 0$,
almost surely. And thus the limiting processes $\lim_{n \to \infty} \widehat{Q}^n_k(\cdot)$ 
are c\`{a}dl\`{a}g and hence the vector valued process $\lim_{n \to \infty} \widehat{Q}^n(\cdot)$ is c\`{a}dl\`{a}g. 
}

\hl{
It remains to show that the finite dimensional distributions of $\lim_{n \to \infty} \widehat{Q}^n(\cdot)$ at time points $t_1,\ldots, t_\ell$ converge to the tuple of zero vectors. From \eqref{eqn:Markov_inq} and \eqref{eq:momentCondMeynDai} we have that  $\Prob\left(|\widehat{Q}^{n}(t)| > \tilde{\varepsilon} \right) \to 0$ as $n \to \infty$. It thus also holds that,
\[
\lim_{n \to \infty}\Prob\left( |\widehat{Q}^{n}(t_1)| > \tilde{\varepsilon}, \dots,  |\widehat{Q}^{n}(t_\ell)| > \tilde{\varepsilon}  \right) = 0,
\]
and this convergence in probability implies convergence in distribution at time points $t_1,\ldots,t_\ell$.
}

Further, assumptions (A1) and (A2) imply that for each class $k$, there are FCLTs for $\hl{\widehat{E}}^n_k$ with diffusion coefficients as described in \eqref{eq:scvFCLTarrivals}. Assumption \hl{(A3)} together with applications of the continuous mapping theorem and \eqref{eq:barTdoesTheJob} imply FCLTs where for each class $k$, $\widetilde{\Phi}^n_{k,.}(t)$ converges weakly to K-dimensional Brownian motion with covariance matrix $\mu_k \frac{\lambda_k}{\mu_k}\Gamma_k$. 

By assumption of mutual independence of primitive processes,  as stated in (A1) and \hl{(A3)}, the limiting covariance matrix of
\[
\left[
\begin{array}{c}
\hl{\widehat{E}}^n(t) \\
\widetilde{\Phi}^n(t)
\end{array}
\right]
\]
is $\Sigma^{(P)}$.  The result then follows from the representation in Lemmas \ref{lem:3} and \ref{lem:4} and the weak convergence of $\widehat{Q}^n$ to $0$.
\qed

We note that Lemma~\ref{lem:2} can also yield diffusion limits for rate allocations. This appears as (7.89), pp.189 in \cite{bookChenYao2001}. In fact, there the authors handle a much wider case in which some queues may be critical and/or overloaded. This is originally from \cite{chen1991stochastic} (6.14), pg 1498.  As stated in the introduction the diffusion limits for $D$ and $\hl{A}$ did not appear in \cite{chen1991stochastic} and subsequent literature. It is insightful to know that we may also obtain joint diffusion limits for $T$ and $D$ or $\hl{A}$, yet we do not pursue this here. Further, handling the case of overloaded queues does also not pose any additional technical difficulty.  The case of critical queues is in general an open question. It was handled in \cite{al2010asymptotic} for the single station queue.

\section{The Zero Service Time View}
\label{sec:zerTime}

In this section we refer to the queues as {\em nodes} to make it clear that there is actually no queueing taking place. For the $\ell$'th customer arriving exogenously first to node $k$,  denote $N_{j|k}(\ell)$ as the number of times that the customer  visits node $j$, and denote $N_{i,j|k}(\ell)$ as the number of times that the customer traverses on the flow $i \to j$. Thus, $N_{j|k}(\ell)=\sum_{i=1}^K N_{i,j|k}(\ell)$. Define now,
\[
\breve{D}_{i,j}(t) := \sum_{k=1}^K \sum_{\ell = 1}^{\hl{E}_k(t)} N_{i,j|k}(\ell),
\quad
\mbox{and}
\quad
\hl{\breve{A}}_{k}(t) := \hl{E}_k(t) + \sum_{i=1}^K \breve{D}_{i,k}(t)
   = \hl{E}_k(t) + \sum_{k'=1}^K \sum_{\ell = 1}^{\hl{E}_k(t)} N_{k|k'}(\ell).
\]
The process $\breve{D}_{i,j}(t)$ is a count of the number of items passing from node $i$ to node $j$ up to time $t$ as if service times are $0$. In particular, the $\ell$'th customer who arrives at node $k$ by time $t$ ($\ell=1, \ldots, \hl{E}_k(t)$) makes an ``instantaneous tour'' through the nodes, passing $N_{i,j|k}(\ell)$ times on the flow $i \rightarrow j$. Similarly, $\hl{\breve{A}}_k(t)$ is the count of the number of jobs arriving to queue $k$ either exogenously or passing through the network assuming that service times are $0$.

By considering both $D(\cdot)$ and $\breve{D}(\cdot)$ on the same probability space, we have that $a.s.$,
\[
D_{i,j}(t) \le \breve{D}_{i,j}(t).
\]
Denote now,
\[
\breve{N}_{i,j}(t) := \breve{D}_{i,j}(t) - D_{i,j}(t).
\]
This is the number of future passes on $i \to j$ by customers that are currently in the system (where service times are generally non-zero) at time $t$. It is obvious from the Markovian nature of the routing that,
\begin{equation}
\label{eq:Nbreve_sums}
\breve{N}_{i,j}(t) \stackrel{d}{=} \sum_{k=1}^K \sum_{\ell = 1}^{Q_k(t)} N_{i,j|k}(\ell),
\end{equation}
where the equality \hl{$\stackrel{d}{=}$} is in distribution and for given $k$,
\[
\big\{\big(N_{i,j|k}(\ell),~i,j \in \{1,\ldots,K\}, i\neq j \big),~\ell=1,2,\ldots \big\},
\]
is an i.i.d. sequence (of $K^2$ dimensional random vectors) whose distribution is induced by a discrete time Markov chain on state space $\{0,1,\ldots,K\}$ with transition matrix,
\[
\widetilde{P} =
\begin{bmatrix}
1 & \mathbf{0}\tr \\
\mathbf{1} - P \mathbf{1} & P 
\end{bmatrix}.
\]
To construct $N_{i,j|k}(\ell)$, denote by $\{X^k_n\}$  a sequence of states generated by the above Markov chain with ${\hl{\Prob}}(X_0=k)=1$ for $k \in \{1,\ldots,K\}$. Then for $i\neq j$, $N_{i,j|k}(\ell)$ \hl{has the same distribution as the random variable}
\[
N_{i,j|k} ~ := ~\sum_{n=1}^\infty \hl{\ind}\{X^k_{n-1}=i,X^k_n=j \},
\]
\hl{with $\ind$ denoting an indicator function. Similarly $N_{j|k}(\ell)$ is distributed as}
\[
N_{j|k} ~ := \sum_{n=0}^\infty \hl{\ind} \{ X^k_n=j   \}.
\]
Since the queueing network is open ($P$ is sub-stochastic), the only recurrent class in this Markov chain is $\{0\}$ and thus the random variables $N_{i,j|k}$ are proper. It is also a standard exercise to show that they have finite mean and variance. 

Denote now,
\[
\breve{\sigma}_{i,j} := \lim_{t \to \infty} \frac{\Cov \Big(\hl{\breve{A}}_i(t),\, \hl{\breve{A}}_j(t) \Big)}{t},
\quad
\mbox{and}
\quad
\breve{\sigma}_{i_1 \to j_1,i_2 \to j_2} := \lim_{t \to \infty} \frac{\Cov \Big(\breve{D}_{i_1,j_1}(t),\, \breve{D}_{i_2,j_2}(t) \Big)}{t}.
\]
As we show now,
these variability parameters (of the zero-service time flows), are the same as the variability parameters of the system with queueing:
\begin{proposition}
\label{prop:NoWaitVar}
If \eqref{eq:momentCondMeynDai} holds, then,
\begin{equation}
\label{eq:sigBrevEqual}
\qquad
\breve{\sigma}_{\hl{i,j}}^2  = {\sigma}_{\hl{i,j}}^2,
\quad
\mbox{and}
\quad
\breve{\sigma}_{i_1 \to j_1,i_2 \to j_2} ={\sigma}_{i_1 \to j_1,i_2 \to j_2}.
\end{equation}
\end{proposition}

\proof{}
We present the proof for the asymptotic variability of $D$, the case of $\hl{A}$ is similar and is omitted. We have,
\begin{eqnarray}
\label{eq:CovBreak}
&&\left|\Cov \Big(\breve{D}_{i_1,j_1}(t), \breve{D}_{i_2,j_2}(t) \Big)  -
\Cov \Big({D}_{i_1,j_1}(t), {D}_{i_2,j_2}(t) \Big) \right| \\  
\nonumber 
&&
   \leq \ 
    \left|\Cov \Big(D_{i_1,j_1}(t), \breve{N}_{i_2,j_2}(t) \Big)\right|
+ \left|\Cov \Big(D_{i_2,j_2}(t), \breve{N}_{i_1,j_1}(t)  \Big)\right| 
+ \left|\Cov \Big( \breve{N}_{i_1,j_1}(t), \breve{N}_{i_2,j_2}(t) \Big)\right|    \\ \nonumber
&&\leq \ 
   \sqrt{\Var \Big(D_{i_1,j_1}(t)\Big) \Var \Big( \breve{N}_{i_2,j_2}(t) \Big)}
+ \sqrt{\Var \Big(D_{i_2,j_2}(t)\Big) \Var \Big( \breve{N}_{i_1,j_1}(t)  \Big)}\\  \nonumber
&&\quad +\sqrt{\Var \Big( \breve{N}_{i_1,j_1}(t)\Big) \Var \Big( \breve{N}_{i_2,j_2}(t) \Big)}.
\end{eqnarray}
For any $(i,j)$ we have that both $\Var (D_{i,j}(t))/t$ and $\Var (\breve{N}_{i,j}(t))$ are bounded from above uniformly in $t$; for the latter this is a consequence of \eqref{eq:momentCondMeynDai}.
Dividing \eqref{eq:CovBreak} by $t$ and taking $t \to \infty$ we get the result.
 
\endproof
Note: a version of the above result also exists for the mean rates, $\lambda$. In this case all that is required is finiteness of the first moments of the queues. 

We now express the components of $\breve{\sigma}$ in terms of $\Expect[ N_{i,j|k}]$ and $\Cov ( N_{i_1,j_1|k}, N_{i_2,j_2|k})$.
\begin{proposition}
\label{prop:sigBrev}
\begin{align*}
\breve{\sigma}_{i_1 \to j_1,i_2 \to j_2}  
&=
 \sum_{k=1}^K \alpha_k \Cov ( N_{i_1,j_1|k}, N_{i_2,j_2|k})+ \sum_{k=1}^K  v_k^2 \Expect[ N_{i_1,j_1|k}]\Expect[ N_{i_2,j_2|k}],\\
\breve{\sigma}_{j_1,j_2}&=
v_{j_1}^2 \Expect[N_{j_2|j_1}] + v_{j_2}^2 \Expect[N_{j_1|j_2}] \\
&+\sum_{k=1}^K \alpha_k \Cov \big( N_{j_1|k}, N_{j_2|k}  \big)\ +\ 
\sum_{k=1}^K v_k^2 \, \Expect[N_{j_1|k}]\,   \Expect[N_{j_2|k}]\\
&=
v_{j_1}^2 \Expect[N_{j_2|j_1}] + v_{j_2}^2 \Expect[N_{j_1|j_2}]\ +\ \sum_{i_1=1}^K \sum_{i_2=1}^K  \breve{\sigma}_{i_1 \to j_1,i_2 \to j_2}.  
\end{align*}
\end{proposition}

\proof{}
We begin with the asymptotic variability of $\breve{D}$, namely $\breve{\sigma}_{i_1 \to j_1,i_2 \to j_2}$ . For illustration we begin with the variance (even though it is a special case of the covariance calculation that follows). Using the conditional variance rule we get,
\[
\Var \Big( \breve{D}_{i,j}(t) \Big) = \sum_{k=1}^K \Var \Big(  \sum_{\ell = 1}^{\hl{E}_k(t)} N_{i,j|k}(\ell) \Big)
=
\sum_{k=1}^K 
\Big(
\Expect[\hl{E}_k(t)]\, \Var ( N_{i,j|k}) + 
\Var(\hl{E}_k(t)) \, \Expect[ N_{i,j|k}]^2
\Big).
\]
Moving onto the covariance, observe that $N_{i_1,j_1|k}(\ell)$ and $N_{i_2,j_2|k'}(\ell)$ are independent whenever $k \neq k'$, hence,
\[
\Cov \Big(\breve{D}_{i_1,j_1}(t),\breve{D}_{i_2,j_2}(t) \Big)
=
\sum_{k=1}^{K}
\Cov
\Big(
\sum_{\ell = 1}^{\hl{E}_k(t)} N_{i_1,j_1|k}(\ell)
,
\sum_{\ell = 1}^{\hl{E}_k(t)} N_{i_2,j_2|k}(\ell)
\Big)
\]
\[
=
\sum_{k=1}^K
\Big(
\Expect[\hl{E}_k(t)]\, \Cov ( N_{i_1,j_1|k}, N_{i_2,j_2|k}) + 
\Var\big(\hl{E}_k(t)\big) \, \Expect[ N_{i_1,j_1|k}]\Expect[ N_{i_2,j_2|k}]
\Big)
\]
where in the second step we use the conditional covariance rule,
\[
\Cov(X,Y) = \Expect\big[\Cov(X,Y|Z)\big] + \Cov\big(\Expect[X|Z],\Expect[Y|Z]\big).
\]
Dividing by $t$ and taking $t \to \infty$ yields the result.

Moving onto the asymptotic variability of $\hl{\breve{A}}$ (this time treating the variance and the other covariance terms together)  we expand and get:
\begin{eqnarray}
\Cov \big(  \hl{\breve{A}}_{j_1}(t),  \hl{\breve{A}}_{j_2}(t) \big)
&=& \nonumber
\sum_{i_2=1}^K \Cov \big( \hl{E}_{j_1}(t), \breve{D}_{i_2, j_2}(t) \big)+
\sum_{i_1=1}^K \Cov \big( \hl{E}_{j_2}(t), \breve{D}_{i_1, j_1}(t) \big)\\
&&+ \label{eq:covEE}
\sum_{i_1=1}^K \sum_{i_2=1}^K \Cov \big( \breve{D}_{i_1, j_1}(t), \breve{D}_{i_2, j_2}(t) \big)
\end{eqnarray}
To rewrite the first sum on the righthand side we can use
\begin{align*}
\Cov \big( \hl{E}_{j_1}(t), \breve{D}_{i_2, j_2}(t) \big)
&=
\sum_{k=1}^K \Cov \big(\hl{E}_{j_1}(t), \sum_{\ell=1}^{\hl{E}_k(t)} N_{i_2,j_2|k}(\ell)  \big)\\
&=
\Expect[\Cov \big(\hl{E}_{j_1}(t), \sum_{\ell=1}^{\hl{E}_{j_1}(t)} N_{i_2,j_2|j_1}(\ell) ~\big|~ \hl{E}_{j_1}(t) \big) \\
&\,\,
\quad+ 
\Cov \big(\Expect[\hl{E}_{j_1}(t) ~|~ \hl{E}_{j_1}(t)],\Expect[\sum_{\ell=1}^{\hl{E}_{j_1}(t)} N_{i_2,j_2|j_1}(\ell) ~\big|~\hl{E}_{j_1}(t)]  \big) \\
&=
\Cov \big( \hl{E}_{j_1}(t),\hl{E}_{j_1}(t) \,\Expect[N_{i_2,j_2|j_1}(\ell)]  \big)\\
&=
\Var \big( \hl{E}_{j_1}(t)  \big) \,\Expect[N_{i_2,j_2|j_1}(\ell)] 
\end{align*}
with a similar expression holding for the second term, while the third term on the right hand side of (\ref{eq:covEE}) can be rewritten using
\begin{align*}
\Cov \big( \breve{D}_{i_1, j_1}(t), \breve{D}_{i_2, j_2}(t) \big)
&=\sum_{k_1=1}^K \sum_{k_2=1}^K  \Cov \big( \sum_{\ell_1=1}^{\hl{E}_{k_1}(t)} N_{i_1,j_1|k_1}(\ell_1),              
                                                                           \sum_{\ell_2=1}^{\hl{E}_{k_2}(t)} N_{i_2,j_2|k_2}(\ell_2),\big)\\     
&=\sum_{k=1}^K  \Expect \left[ \sum_{\ell=1}^{\hl{E}_{k}(t)} \Cov \big(N_{i_1,j_1|k}(\ell),              
                                                                                                          N_{i_2,j_2|k}(\ell) \big) \right]\\
&\quad+   \sum_{k=1}^K  \Cov \big( \hl{E}_{k}(t)~ \Expect[N_{i_1,j_1|k}(\ell)],~ \hl{E}_k(t)~ \Expect[N_{i_2,j_2|k}(\ell)] \big)\\
&=\sum_{k=1}^K \Expect[\hl{E}_k(t)] \Cov \big(N_{i_1,j_1|k}(\ell),~ N_{i_2,j_2|k}(\ell)  \big) \\
&\quad+\hl{\sum_{k=1}^K } \Var \big(\hl{E}_k(t)  \big)~ \Expect[N_{i_1,j_1|k}(\ell)]~  \Expect[N_{i_2,j_2|k}(\ell)].
\end{align*}
where we used the independence of different customers in the absence of queuing.

Substituting in (\ref{eq:covEE}) and using $\sum_{i=1}^K N_{i,j|k}(\ell)=N_{j|k}(\ell)$ we arrive at
\begin{eqnarray*}
\Cov \big(  \hl{\breve{A}}_{j_1}(t),  \hl{\breve{A}}_{j_2}(t) \big)&=& 
\Var \big( \hl{E}_{j_1}(t)  \big)\, \Expect[N_{j_2|j_1}(\ell)] + 
\Var \big( \hl{E}_{j_2}(t)  \big) \, \Expect[N_{j_1|j_2}(\ell)] \\
&&+
\sum_{k=1}^K \Expect[\hl{E}_k(t)] \,\Cov \big( N_{j_1|k}(\ell), N_{j_2|k}(\ell)  \big) + 
\Var \big(\hl{E}_k(t)  \big) \, \Expect[N_{j_1|k}(\ell)]  \Expect[N_{j_2|k}(\ell)].
\end{eqnarray*}
Now dividing by $t$ and letting $t\rightarrow \infty$, the result is immediate. 
\endproof

We now represent $\Expect[ N_{i,j|k}]$ and $\Cov ( N_{i_1,j_1|k}, N_{i_2,j_2|k})$ in terms of the routing matrix $P$. It is an elementary application of ``first step analysis'' to calculate the desired moments (c.f. \cite{karlin1975fcs} and/or \cite{kemeny1960fmc}), yet we have not seen this specific calculation elsewhere, so we spell out the details. Define:
\[
m(i,j) := 
\left[
\begin{array}{c}
\Expect [ N_{i,j|1}] \\
\vdots \\
\Expect [ N_{i,j|K}] \\
\end{array}
\right],
\quad
m(i_1,j_1,i_2,j_2) := 
\left[
\begin{array}{c}
\Expect [ N_{i_1,j_1|1} \, N_{i_2,j_2|1}] \\
\vdots \\
\Expect [N_{i_1,j_1|K} \,N_{i_2,j_2|K}] \\
\end{array}
\right],
\]
\[
c(i_1,j_1,i_2,j_2) := 
\left[
\begin{array}{c}
\Cov \big( N_{i_1,j_1|1},N_{i_2,j_2|1}\big) \\
\vdots \\
\Cov \big(N_{i_1,j_1|K},N_{i_2,j_2|K}\big) \\
\end{array}
\right].
\]
\begin{lemma}
\label{lemma:firstStep}
The definition of $m(i,j)$ in \eqref{eq:MijFirst} agrees with the above, namely
\[
m(i,j) = (I- P)^{-1} e_{i,i} P_{\cdot,j}.
\]

Further, let $i_1 \rightarrow j_1$ and $i_2 \rightarrow j_2$ be distinct flows (i.e., $i_1 \neq i_2$,  or $j_1 \neq j_2$, or both), then
\begin{eqnarray}
\label{eq:mijij}
m(i_1,j_1,i_2,j_2) 
&=&
m(i_1,j_1) m_{j_1}(i_2,j_2) + m(i_2,j_2) m_{j_2}(i_1,j_1),\\
m(i,j,i,j) 
&=&
m(i,j) \big( 1+2 m_{j}(i,j)\big), \nonumber
\end{eqnarray}
and thus,
\begin{eqnarray}
\label{eq:cijij}
c(i_1,j_1,i_2,j_2) &=& m(i_1,j_1) \, m_{j_1}(i_2,j_2) + m(i_2,j_2) \, m_{j_2}(i_1,j_1) - m(i_1,j_1) \bullet m(i_2,j_2),\\
c(i,j,i,j) &=& m(i,j) \big(1+2 m_{j}(i,j)\big) - m(i,j) \bullet m(i,j). \nonumber 
\end{eqnarray}
\end{lemma}
\proof{}
It is well-known that $\Expect[N_{i|k}]$ is the $(k,i)$th element of $(I-P)^{-1}$, and clearly $\Expect[N_{i,j|k}]=\Expect[N_{i|k}]\,p_{i,j}$, from which the first statement follows. For $\Expect[N_{i_1,j_1|k}\, N_{i_2,j_2|k}]$ we condition on the first transition from the initial node $k$, as follows (let $i_1 \neq i_2$,  and/or $j_1 \neq j_2$).
\begin{eqnarray}
\nonumber
\Expect[N_{i_1,j_1|k}~N_{i_2,j_2|k}] &=&
\sum_{k'=1, k' \not\in \{j_1,j_2\}}^K p_{k , k'} ~ \Expect[N_{i_1,j_1|k'}N_{i_2,j_2|k'}] + 
\\
\nonumber
&&p_{k,j_1} \Expect [ (\delta_{k,i_1} + N_{i_1,j_1|j_1}) N_{i_2,j_2|j_1}]
+ p_{k,j_2} \Expect [ N_{i_1,j_1|j_2} (\delta_{k,i_2} + N_{i_2,j_2|j_2})] \\
\nonumber
&=&
\sum_{k'=1}^K p_{k , k'} ~ \Expect[N_{i_1,j_1|k'}\, N_{i_2,j_2|k'}]  + \\
\label{eq:NijCovariance}
&& p_{k,j_1} \delta_{k,i_1} \Expect [ N_{i_2,j_2|j_1}]  + p_{k,j_2} \delta_{k,i_2} \Expect [ N_{i_1,j_2|j_2}].
\end{eqnarray}

The equations \eqref{eq:NijCovariance} can be represented as,
\[
m(i_1,j_1,i_2,j_2) = P\,  m(i_1,j_1,i_2,j_2) + e_{i_1,i_1} P_{\cdot,j_1} \, m_{j_1}(i_2,j_2) +
e_{i_2,i_2} P_{\cdot,j_2} \, m_{j_2}(i_1,j_1).
\]
or rearranged to,
\begin{eqnarray*}
\nonumber
m(i_1,j_1,i_2,j_2) &=& (I-P)^{-1} \, \big(
  e_{i_1,i_1} \,P_{\cdot,j_1} \,m_{j_1}(i_2,j_2) +
e_{i_2,i_2}\, P_{\cdot,j_2}\, m_{j_2}(i_1,j_1)
\big),
\end{eqnarray*}
which yields \eqref{eq:mijij}. In a similar way, we can show that 
\begin{eqnarray}
\nonumber
\Expect[N^2_{i,j|k}] &=&
\sum_{k'=1}^K p_{k , k'} ~ \Expect[N^2_{i,j|k'}]  + 
p_{k,j} \delta_{k,i} \big(1+2 \Expect [ N_{i,j|j}] \big),
\end{eqnarray}
which gives 
\begin{eqnarray*}
\nonumber
m(i,j,i,j) &=& (I-P)^{-1} 
  e_{i,i} P_{\cdot,j} \big(1+2 m_{j}(i,j)
\big). 
\end{eqnarray*}
\endproof

{\em Proof of Theorem~\ref{thm:main} (iii):}  Proposition~\ref{prop:NoWaitVar} indicates that under property \eqref{eq:momentCondMeynDai}, the variability parameters are the same as those of the zero service time processes, and \eqref{eq:momentCondMeynDai} follows from the theorem assumptions. Now the combination of Proposition~\ref{prop:sigBrev} and Lemma~\ref{lemma:firstStep} yield the result.
\qed

\section{Asymptotic Variance and Uniform Integrability}
\label{sec:asympVarAndUI}

As stated at onset our original goal is to obtain expressions for ${\sigma}_{k_1,k_2}$ and ${\sigma}_{i_1 \to j_1, i_2 \to j_2}$. As we state in Theorem~\ref{thm:main} (ii) these can now be read off from the matrices $\Sigma^{(\hl{A})}$ and $\Sigma^{(D)}$ respectively. The presentation in this section is for the ${\sigma}_{i_1 \to j_1, i_2 \to j_2}$ terms; analogous results for the terms associated with $\hl{A}(\cdot)$ can be proved in the exact same manner.  

Proving Theorem~\ref{thm:main} (ii) requires establishing suitable uniform integrability (UI) conditions for the following families:
\begin{eqnarray*}
{\cal D}_{i,j}^{(1)} & = & \left\{ \frac{ D_{i,j}(t) - \lambda_{i,j}t}{\sqrt{t}},\,\, t\ge t_0 \right\} ,\\
{\cal D}_{i,j}^{(2)} &=&
\left\{  \frac{ \big( D_{i,j}(t) - \lambda_{i,j}t \big)^2}{t},\,\, t\ge t_0 \right\}, \\
{\cal D}_{(i_1,j_1),(i_2,j_2)} 
&=&\left\{\frac{ \big( D_{i_1,j_1}(t) - \lambda_{i_1,j_1}t \big) \big( D_{i_2,j_2}(t) - \lambda_{i_2,j_2}t \big)}{t},\,\, t\ge t_0 \right\},
\end{eqnarray*}
where $t_0>0$ is arbitrary. Note that while each of the families ${\cal D}_{i,j}^{(2)}$ is a special case of ${\cal D}_{(i_1,j_1),(i_2,j_2)}$, we treat it separately in this section for clarity.  See for example \cite{gut2005probability} for properties of UI sequences and families, and relations to weak convergence. 

The following proposition relates the diffusion parameters to the asymptotic variance parameters.
\begin{proposition}
\label{prop:asympVarDiffEqual}
If ${\cal D}_{i,j}^{(1)}$ and ${\cal D}_{i,j}^{(2)}$ are UI then,
\[
{\sigma}_{i \to j}^2 = \Sigma^{(D)}_{(i-1)K+j,\,(i-1)K+j}~.
\]
If ${\cal D}_{i,j}^{(1)}$ and ${\cal D}_{(i_1,j_1),(i_2,j_2)} $ are UI then,
\[
 \quad {\sigma}_{i_1\to j_1,i_2 \to j_2} = \Sigma^{(D)}_{(i_1-1)K+j_1,(i_2-1)K + j_2}~.
\]
\end{proposition}
\proof{}
By the projection map at time $t=1$ (c.f. \cite{bookWhitt2001}) we have the convergence in distribution:
\[
\frac{D_{i,j}(t) - \lambda_{i,j}t}{\sqrt{t}} \Rightarrow \widehat{D}_{i,j}(1).
\]
Further, using the continuous mapping theorem we obtain,
\[
\frac{\big( D_{i,j}(t) - \lambda_{i,j}t \big)^2}{t} \Rightarrow \big(\widehat{D}_{i,j}(1)\big)^2.
\]
Similarly we have the convergence in distribution on ${\mathbb R}^2$:
\[
\Big[
\frac{D_{i_1,j_1}(t) - \lambda_{i_1,j_1}t}{\sqrt{t}}
,\,\,
\frac{D_{i_2,j_2}(t) - \lambda_{i_2,j_2}t}{\sqrt{t}}
\Big]
 \Rightarrow \Big[\widehat{D}_{i_1,j_1}(1),\,\, \widehat{D}_{i_2,j_2}(1) \Big],
\]
and thus using the continuous mapping theorem,
\[
\frac{D_{i_1,j_1}(t) - \lambda_{i_1,j_1}t}{\sqrt{t}}
\,\cdot \,
\frac{D_{i_2,j_2}(t) - \lambda_{i_2,j_2}t}{\sqrt{t}}
\Rightarrow 
\widehat{D}_{i_1,j_1}(1)\, \cdot \, \widehat{D}_{i_2,j_2}(1).
\]
Under the UI conditions established below the above weak convergences in distribution imply that,
\begin{eqnarray*}
\lim_{t \to \infty} \Expect\Big[\frac{ D_{i,j}(t) - \lambda_{i,j}t}{\sqrt{t}}\Big] &=& \Expect\big[\widehat{D}_{i,j}(1)\big],\\
\lim_{t \to \infty} \Expect\Big[\frac{ (D_{i,j}(t) - \lambda_{i,j}t)^2}{t}\Big] &=& \Expect\big[\big(\widehat{D}_{i,j}(1)\big)^2\big],
\end{eqnarray*}
as well as,
\begin{equation*}
\lim_{t \to \infty} \Expect\Big[\frac{\big( D_{i_1,j_1}(t) - \lambda_{i_1,j_1}t \big)}{\sqrt{t}}
\cdot \frac{\big( D_{i_2,j_2}(t) - \lambda_{i_2,j_2}t \big)}{\sqrt{t}}\Big] = \Expect[
\widehat{D}_{i_1,j_1}(1)\,\cdot\,\widehat{D}_{i_2,j_2}(1) ].
\end{equation*}
Combining this implies that,
\begin{eqnarray*}
{\sigma}_{i \to j}^2 &=& \lim_{t \to \infty} \frac{\Var \big( D_{i,j}(t) \big)}{t}
=
\lim_{t \to \infty} \frac{\Var \big( D_{i,j}(t) - \lambda_{i,j} t \big)}{t} \\
&=&
\lim_{t \to \infty}  \frac{ \Expect [(D_{i,j}(t) - \lambda_{i,j} t)^2]}{t}
-
\Big(
\lim_{t \to \infty}  \frac{ \Expect [D_{i,j}(t) - \lambda_{i,j} t]}{\sqrt{t}}
\Big)^2 \\
&=& 
\Expect[
\big(\widehat{D}_{i,j}(1)\big)^2] - \big( \Expect[ \widehat{D}_{i,j}(1)] \big)^2
= \Var \big( \widehat{D}_{i,j}(1) \big) = \Sigma^{(D)}_{(i-1)K+j,\,(i-1)K+j}.
\end{eqnarray*}
Similarly,
\begin{eqnarray*}
{\sigma}_{i_1 \to j_1,i_2 \to j_2} &=& \lim_{t \to \infty} \frac{ \Cov\big( D_{i_1,j_1}(t), D_{i_2,j_2}(t) \big)}{t}\\
&=&
\lim_{t \to \infty} \frac{\Cov\big( D_{i_1,j_1}(t) - \lambda_{i_1,j_1} t, D_{i_2,j_2}(t) - \lambda_{i_2,j_2} t \big)}{t} \\
&=&
\lim_{t \to \infty}  \frac{ \Expect [(D_{i_1,j_1}(t) - \lambda_{i_1,j_1} t)(D_{i_2,j_2}(t) - \lambda_{i_2,j_2} t)]}{t}
\\
&&\hspace{2cm} 
-\Big(
\lim_{t \to \infty}  \frac{ \Expect [D_{i_1,j_1}(t) - \lambda_{i_1,j_1} t]}{\sqrt{t}}
\Big)
\Big(
\lim_{t \to \infty}  \frac{ \Expect [D_{i_2,j_2}(t) - \lambda_{i_2,j_2} t]}{\sqrt{t}}
\Big)
 \\
&=& \Cov \Big( \widehat{D}_{i_1,j_1}(1),\,\widehat{D}_{i_2,j_2}(1) \Big) 
= \Sigma^{(D)}_{(i_1-1)K+j_1,\,(i_2-1)K+j_2}. 
\end{eqnarray*}
\endproof
In establishing the UI, we make use of the following useful inequality: 
For $r>1$ and arbitrary real values $z_1,\ldots,z_K$,
\begin{equation}
\label{eq:powerIneq}
\Big|\sum_{k=1}^K z_k \Big|^r \le K^{r-1} \sum_{k=1}^K \big|z_k \big|^r,
\end{equation}
which is a simple consequence of Jensen's inequality.
We now establish the required  UI.

\begin{proposition}
\label{prop:UI1}
If \eqref{eq:barTdoesTheJob} and \eqref{eq:AUIass} hold,
then the families of random variables ${\cal D}_{i,j}^{(1)}$, ${\cal D}_{i,j}^{(2)}$ and ${\cal D}_{(i_1,j_1),(i_2,j_2)}$ are UI.
\end{proposition}
\proof{}
We first note that UI of ${\cal D}_{i,j}^{(2)}$ implies UI of the other two types of families as well, due to Theorem~4.7  (with $p=q=2$) in Chapter~5 of \cite{gut2005probability}. To establish UI of ${\cal D}_{i,j}^{(2)}$ , recall from the previous section the representation $D_{i,j}(t) = \breve{D}_{i,j}(t) - \breve{N}_{i,j}(t)$, where $\breve{D}_{i,j}(t)$ is the number of instantaneous passes on flow $i\rightarrow j$, and $\breve{N}_{i,j}(t)$ is the number of future passes on that flow. By applying  \eqref{eq:powerIneq} for $r=1$ and $r=2$, respectively,  we have:
\begin{eqnarray*}
\Big| \frac{ D_{i,j}(t) - \lambda_{i,j}t}{\sqrt{t}} \Big| 
& \le &
\Big| \frac{ \breve{D}_{i,j}(t) - \lambda_{i,j}t}{\sqrt{t}} \Big| + \Big| \frac{\breve{N}_{i,j}(t)}{\sqrt{t}} \Big|, \\
\Big|  \frac{ \big( D_{i,j}(t) - \lambda_{i,j}t \big)^2}{t} \Big| 
& \le &
2
\Big(
\Big| \frac{ \breve{D}_{i,j}(t) - \lambda_{i,j}t}{\sqrt{t}} \Big|^2 + \Big| \frac{\breve{N}_{i,j}(t)}{\sqrt{t}} \Big|^2
\Big).
\end{eqnarray*}
It thus suffices to show that,
\[
\breve{{\cal D}}_{i,j}^{(2)} :=\left\{  \frac{ \big( \breve{D}_{i,j}(t) - \lambda_{i,j}t \big)^2}{t},\,\, t\ge t_0 \right\},
\quad
\mbox{and}
\quad
\breve{{\cal N}}_{i,j}^{(2)}
:=
\left\{  \frac{ \big( \breve{N}_{i,j}(t) \big)^2}{t},\,\, t\ge t_0 \right\},
\]
are UI.

To see $\breve{{\cal D}}_{i,j}^{(2)}$ is UI it is useful to denote,
\[
\breve{D}_{i,j|k}(t) := \sum_{\ell=1}^{\hl{E}_k(t)} N_{i,j|k}(\ell)
\quad
\mbox{and}
\quad
\lambda_{i,j|k} := \alpha_k \Expect[N_{i,j|k}].
\]
Note that since, $\breve{D}_{i,j}(t) = \sum_{k=1}^K\breve{D}_{i,j|k}(t)$, we have $\sum_{k=1}^K \lambda_{i,j|k} = \lambda_{i,j}$. We now get,
\begin{align*}
\left|\frac{ \big( \breve{D}_{i,j}(t) - \lambda_{i,j}t \big)^2}{t}\right|
&=
\left|\frac{ \big( \sum_{k=1}^K\breve{D}_{i,j|k}(t) - (\sum_{k=1}^K\lambda_{i,j|k})t \big)^2}{t}\right|\\
&=
\left(\frac{\big|  \sum_{k=1}^K \big(\breve{D}_{i,j|k}(t) - \lambda_{i,j|k}t \big)\big|}{\sqrt{t}} \right)^2\\
&\le
K \sum_{k=1}^K\left| \frac{  \breve{D}_{i,j|k}(t) - \lambda_{i,j|k}t }{\sqrt{t}} \right|^2.
\end{align*}
In the above we again used \eqref{eq:powerIneq} with $r = 2$. We now need to show that the families
\[
\left\{  \frac{ \big( \breve{D}_{i,j|k}(t) - \lambda_{i,j|k}t \big)^2}{t},\,\, t\ge t_0 \right\},
\]
are UI:
{\small
\begin{align*}
\frac{ \big( \breve{D}_{i,j|k}(t) - \lambda_{i,j|k}t \big)^2}{t} 
& =
\frac{ \big(\sum_{\ell=1}^{\hl{E}_k(t)} N_{i,j|k}(\ell) - \lambda_{i,j|k} t\big)^2}{t} \\
&= 
\frac{ \big(\sum_{\ell=1}^{\hl{E}_k(t)+1} N_{i,j|k}(\ell) - \lambda_{i,j|k} t  - N_{i,j|k}(\hl{E}_k(t)+1)\big)^2}{t} \\
&= 
\frac{ \big(\sum_{\ell=1}^{\hl{E}_k(t)+1} \big(N_{i,j|k}(\ell)  - \frac{\lambda_{i,j|k}}{\alpha_k} \big)  + (\hl{E}_k(t)+1)\frac{\lambda_{i,j|k}}{\alpha_k}- \lambda_{i,j|k} t  - N_{i,j|k}(\hl{E}_k(t)+1)\big)^2}{t} \\
\displaybreak
&=
\frac{ \big(\sum_{\ell=1}^{\hl{E}_k(t)+1} \big(N_{i,j|k}(\ell)  - \frac{\lambda_{i,j|k}}{\alpha_k} \big)  + ((\hl{E}_k(t)+1) - \alpha_k t) \frac{\lambda_{i,j|k}}{\alpha_k}  - N_{i,j|k}(\hl{E}_k(t)+1)\big)^2}{t} \\
&\le
3
\left(
\frac{ \big(\sum_{\ell=1}^{\hl{E}_k(t)+1} \big(N_{i,j|k}(\ell)  - \frac{\lambda_{i,j|k}}{\alpha_k} \big)  \big)^2}{t} \right.\\
                   &\hspace{2cm}\left. + \frac{\big( ((\hl{E}_k(t)+1) - \alpha_k t) \frac{\lambda_{i,j|k}}{\alpha_k} \big)^2}{t}  + \frac{\big( N_{i,j|k}(\hl{E}_k(t)+1)\big)^2}{t}
\right).
\end{align*}
}

The first term is a stopped random walk with zero mean increments where $\hl{E}_k(t)+1$ is UI by \eqref{eq:AUIass}. Thus due to Theorems 6.1--6.3 in \cite{gut2009stopped}, the first term is UI. The second term is UI again by \eqref{eq:AUIass}.  The third term is obviously UI since the family $N_{i,j|k}(\cdot)$ is i.i.d.

To show that $\breve{{\cal N}}_{i,j}^{(2)}$ is UI, we need to show that the second moment of $\breve{{N}}_{i,j}(t)/\sqrt{t}$ converges (to zero). This approach is due to Remark~5.4 in Chapter~5 of  \cite{gut2005probability}.
Define $\breve{{N}}^Q_{i,j|k}(t) :=\sum_{\ell = 1}^{Q_k(t)} N_{i,j|k}(\ell)$, where $Q_k(t)$ is the queue length at node $k$ at time $t$. Then the expectation and variance of the random sums $\breve{{N}}^Q_{i,j|k}(t)$, and hence also (by (\ref{eq:Nbreve_sums})) of $\breve{{\cal N}}_{i,j}(t)$, can be expressed in the expectations and variances of $Q_k(t)$ and $N_{i,j|k}(\ell)$, all of which are $O(1)$ by \eqref{eq:barTdoesTheJob}.
Thus the result follows. 
\endproof

{\em Proof of Theorem~\ref{thm:main} (ii):} Proposition~\ref{prop:UI1} relies on \eqref{eq:barTdoesTheJob} and \eqref{eq:AUIass} which follow from the assumptions of the theorem. Proposition~\ref{prop:UI1} then establishes UI of the families needed for Proposition~\ref{prop:asympVarDiffEqual} which exactly states (ii).
\qed

\section{Numerical Example}
\label{sec:example}
\begin{figure}[h]
{\includegraphics[width=3.5in]{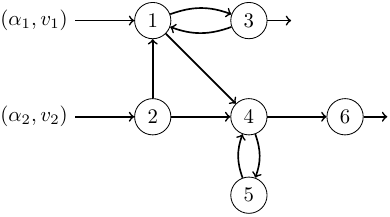}}
\caption{Example network.} 
\label{fig1}
\end{figure}
Consider the $6$ node network illustrated on Figure~\ref{fig1} with parameters,
\[
P=
\left[
\begin{array}{cccccc}
0  & 0 & 1/2  & 1/2  & 0  & 0   \\
1/2  & 0  & 0  & 1/2  & 0  & 0   \\
1/2 & 0  & 0  & 0  & 0  & 0   \\
0  & 0  & 0  & 0  & 1/2  & 1/2   \\
0  & 0  & 0  & 1  & 0  & 0  \\
0  & 0  & 0  & 0  & 0  & 0   \\
\end{array}
\right]
,
\quad
\mu = 
\left[
\begin{array}{c}
8.25 \\
8.25 \\
5 \\
8.25 \\
5 \\
5 
\end{array}
\right],
\]
\begin{equation}
\label{eq:exampleParams}
\alpha = 
\left[ 1 \,\,\,\, 4 \,\,\,\, 0 \,\,\,\, 0 \,\,\,\, 0 \,\,\,\, 0 \right]\tr,\, \, \text{ and }
\quad
v^2 = \left[2 \,\,\,\, 2 \,\,\,\, 0 \,\,\,\, 0 \,\,\,\, 0  \,\,\,\, 0 \right]\tr.
\end{equation}
For this network,
\[
\lambda = (I-P\tr)^{-1} \alpha = \Big[4\,\,\,\, 4 \,\,\,\, 2 \,\,\,\, 8 \,\,\,\, 4 \,\,\,\, 4 \Big]\tr < \mu.
\]
Hence if considered as a single class network, assumptions \hl{(A1) — (A6)} hold and the network is stabilized by any work conserving policy.
Note that besides verification of the above inequality, the values of $\mu$ do not play a further role in the calculation of the variability parameters. Nevertheless, we use them in a simulated example  below.

It is now a straight forward matter to use \eqref{eq:covExpression} (or alternatively \eqref{eq:mainThmiii-1}--\eqref{eq:mainThmiii-2}) from our main theorem to obtain variability parameters. Note that in this process, the only matrix that requires inversion is $(I-P\tr)$. The rest of the calculations follow from matrix composition, addition and multiplication operations.

The resulting matrix $\Sigma^{(D)}$ is of dimension $36 \times 36$. We present the diagonals of this matrix (which are  $\sigma^2_{i\to j}$) in Table~\ref{tab1}.

\begin{table}[h]
{Covariance values of flows $\boldsymbol{i \to j}$. \\ \label{tab1}}
{\begin{tabular}{r|cccccc}
$i \backslash j $    &   1	&	2	&	3	&	4	&	5	&	6 \\
\hline
1    &   0	&	0	&	$32/9$	&	$20/9$	&	0	&	0 \\
2    &   $3/2$	&	0	&	0	&	$3/2$	&	0	&	0 \\
3    &   $31/18$	&	0	&	0	&	0	&	0	&	0 \\
4    &   0	&	0	&	0	&	0	&	$199/18$	&	$55/18$ \\
5    &   0	&	0	&	0	&	$199/18$	&	0	&	0 \\
6    &   0	&	0	&	0	&	0	&	0	&	0 \\
\end{tabular}}
{}
\end{table}

As a further illustration we present a few selected non-diagonal elements of $\Sigma^{(D)}$:
\[
\sigma_{2\to 1, 2 \to 4} = -1/2,
\quad
\sigma_{4\to 5, 5 \to 4} = 199/18,
\quad
\sigma_{1\to 3, 4 \to 6} = 5/9,
\quad
\sigma_{1\to 3, 2 \to 4} = -1/3.
\]
In discussing these values, it is good to consider the {\em  asymptotic correlation coefficient}:
\[
r_{i_1 \to i_2,  j_1 \to j_2} : = 
\frac{\sigma_{i_1\to i_2, j_1 \to j_2}}{\sqrt{\sigma_{i_1\to i_2}^2 \sigma_{j_1\to j_2}^2}}.
\]
For these selected flow pairs it evaluates to
\[
r_{2\to 1, 2 \to 4} = -\frac{1}{3},
\quad
r_{4\to 5, 5 \to 4} = 1,
\quad
r_{1\to 3, 4 \to 6}  \approx 0.16856,
\quad
r_{1\to 3, 2 \to 4} \approx -0.14434.
\]

The first two values are easily explained in our example, the other two are not. For $r_{2\to 1, 2 \to 4}$ consider the Bernoulli splitting at the output of queue $2$ and the fact there is no feedback to this queue. Recall that in this case $\sigma_{2\to 1, 2 \to 4} = (v_2^2-\alpha_2)/4$ for $v_2^2=2, \alpha_2=4$. In this case the asymptotic correlation coefficient is $(v_2^2-\alpha_2)/(v_2^2+\alpha_2)$. In considering $r_{4\to 5, 5 \to 4}$ observe that there is no random routing in this part of the network: All jobs that enter $5$ come from $4$ and then return to $5$. 

We are not aware of an ``easy'' explanation of the values of $r_{1\to 3, 4 \to 6}$ and $r_{1\to 3, 2 \to 4}$.  It is insightful to see that as in this case, some correlations between flows are positive while others are negative. We do not know of an a priori way of finding out the sign of these correlations without using our main result. In fact, evaluating $\Sigma^{(D)}$ with $v_2$ as free variable, we get,
\[
r_{1\to 3, 2 \to 4} = \frac{v_2^2-4}{\sqrt{(v_2^2+4)(v_2^2+30)}}.
\]
We thus see that the sign of the correlation between those two flows depends on the variability of the arrival process into $2$. Observe that in the asymptotically uncorrelated case (i.e. when $v_2 = 4$),
\[
\lim_{t \to \infty} \frac{\Var\big(\hl{E}_2(t) \big)}{\Expect[\hl{E}_2(t)]} = 1,
\]
as is for a Poisson process.  This is consistent with the fact that in the case of a classic Jackson network (Poisson arrival process and exponential processing times) case, since node $2$ has no feedback its output is a Poisson process and splitting of departures from node $2$ results in two independent Poisson flows, $2\to 1$ and $2 \to 4$.  The first of these flows affects $1 \to 3$ but not the second. Hence in such a case it is expected that $r_{1\to 3, 2 \to 4} = 0$.

\subsection*{Arrivals to Individual Queues}

Moving onto arrival processes into individual queues, application of our main result yields:
\begin{align}
\label{eqn:SigmaE}
\Sigma^{(\hl{A})} =
\left[
\begin{array}{cccccc}
68/9   &	4/3	&	40/9	&	44/9	&	22/9	&	22/9 \\
   &	2	&	2/3	&	10/3	&	5/3	&	5/3 \\
   &		&	32/9	&	10/9	&	5/9	&	5/9 \\
   &		&		&	182/9	&	127/9	&	55/9 \\   
   &		&		&		&	199/18	&	55/18 \\      
   &		&		&		&		&	55/18 \\         
\end{array}
\right].
\end{align}
Observe that $\sigma^2_2=2$ as expected since there are only exogenous arrivals to this queue. Further since all jobs that pass through queue $5$ eventually also pass through queue $6$ we have,
\[
\sigma_{k,5} =\sigma_{k,6}, \,\,\,\, k= 1,2,3,6.
\]

It is the diagonal elements of $\Sigma^{(\hl{A})}$ that may be useful for network decomposition approximations (which we do not explore further in this paper). Normalizing the diagonals by $\lambda$ we get,
\begin{equation}
\label{eq:SCVvals1}
c^2 =
\left[
\begin{array}{cccccc}
1.89 &
0.5 &
1.78 &
2.53 &
2.76 &
0.76
\end{array}
\right]\tr.
\end{equation}

\subsection*{Comparison with the Innovations Method}
\label{sec:example}
To compare our method with the {\it innovations method} of \cite{kim2011modeling} (see also \cite{KimRaviColm2005}), we now describe how to compute the asymptotic covariance terms of arrival processes, and the asymptotic variability parameters of flows using the innovations method. Towards that end we outline a step by step procedure to establish equation (42) of \cite{kim2011modeling} under the assumption that all the service times are zero; that is, $\rho_i = 0$, for all $i = 1,2,\dots, K$ (using the notation of \cite{kim2011modeling}). 

Let $\hl{E}^{(0)}_i(t) := \hl{E}_i(t) - \Expect\left[ \hl{E}_i(t)\right]$ be the normalized arrival process into the queue $i$. Similarly, define $\hl{A}^{(0)}_i(t)$ and $D^{(0)}_{i,j}(t)$ using $\hl{A}_i(t)$ and $D_{i,j}(t)$, respectively. Further, let $X(t)$ and $Y(t)$ be two $K^2 + K$ dimensional column vectors defined by
\[
X(t) = \begin{bmatrix}
				\hl{E}^{(0)}(t) \\
				\zeta(t)
             \end{bmatrix},            
             \qquad 
             \text{ and }
             \qquad
Y(t) = \begin{bmatrix}
				\hl{A}^{(0)}(t) \\
				D^{(0)}(t)
             \end{bmatrix},
\]
where 
\begin{align*}
\hl{E}^{(0)}(t) &= \left[ \hl{E}^{(0)}_1(t), \dots, \hl{E}^{(0)}_K(t)\right]\tr, \\
\hl{A}^{(0)}(t) &= \left[ \hl{A}^{(0)}_1(t), \dots, \hl{A}^{(0)}_K(t)\right]\tr, \\
D^{(0)}(t) &= \left[D^{(0)}_{1,1}(t), \dots, D^{(0)}_{1,K}(t), D^{(0)}_{2,1}(t), \dots, D^{(0)}_{K,1}(t), \dots, D^{(0)}_{K,K}(t) \right]\tr,
\end{align*}
and $\zeta(t) = \left[\zeta_{1,1}(t), \dots,  \zeta_{1,K}(t), \zeta_{2,1}(t), \dots, \zeta_{K,1}(t), \dots, \zeta_{K,K}(t)\right]\tr$ is a column vector of so-called innovation processes assumed to have the following properties:
\begin{align*}
\Expect[\zeta_{i,j}(t)] &= 0,\\
\Cov\left( \zeta_{i, j}, \, \zeta_{i, k}\right) &= p_{i,j} (\delta_{j,k} - p_{i,k}) \Expect\left[ \hl{A}_{i}(t)\right],\\
\Cov\left(\hl{A}_i(t),\, \zeta_{i, j}(t)\right) &= 0,\\
\Cov\left(\hl{E}_i(t),\, \zeta_{j, k}(t)\right) &= 0\,\, \text{ for all }\,\, i, j \text{ and } k,\\
\Cov\left( \zeta_{i_1, j_1}, \, \zeta_{i_2, j_2}\right) &= 0 \,\, \text{ for }\, i_1 \neq i_2, \text{ and for all } j_1, j_2 .
\end{align*}
See Section~4 of \cite{kim2011modeling} for more details on the innovations. 

With these processes and assumptions at hand, \cite{kim2011modeling} presents the following (adapted here to our notation):
Let 
\[
F = \begin{bmatrix}
    0 & B \\
    P_c & 0
    \end{bmatrix},
\]
Now using $\rho_i = 0$, for all $i = 1,2,\dots, K$, in \cite{kim2011modeling},  equation (35) of \cite{kim2011modeling} can be re-expressed as 
\[
Y(t) = F\, Y(t) +  X(t),
\]
or equivalently, $Y(t) = (I - F)^{-1}\, X(t)$. 
Non-singularity of $I - F$ follows from that of $I - P\tr$. To see this, first note that  
\[
I - F = \begin{bmatrix}
    I_K & -B \\
    -P_c & I_{K^2}
    \end{bmatrix}.\,\,
\]
Using {\em Banachiewicz inversion formula} for the inverse of a partitioned matrix (see for example \cite{Zhang06Schur}), we have
\begin{align*}
(I - F )^{-1} &= \begin{bmatrix}
    (I_K - B P_c)^{-1}& (I_K - B P_c)^{-1} B \\
    P_c (I_K - B P_c)^{-1} & I_{K^2} + P_c (I_K - B P_c)^{-1} B
    \end{bmatrix} \\
    &= \begin{bmatrix}
    (I_K - P\tr)^{-1}& (I_K - P\tr)^{-1} B \\
    P_c (I_K - P\tr)^{-1} & I_{K^2} + P_c (I_K - P\tr)^{-1} B
    \end{bmatrix}\\
    &= \begin{bmatrix}
    G \\
    H
    \end{bmatrix},
\end{align*}
where the second equality follows from the fact that $B\, P_c = P\tr$ and the last equality follows from the definitions of $G$ and $H$, which are given by \eqref{eqn:G_def} and \eqref{eqn:H_def}, respectively. Hence the claim that $I - F$ is non-singular follows from the assumption that $I - P\tr$ is non-singular.  

Further, we have that $\Expect\left[Y(t)\, Y(t)\tr\right] =  \begin{bmatrix}
    G \\
    H
    \end{bmatrix}\, \Expect\left[X(t)\, X(t)\tr\right]\,  \begin{bmatrix}
    G\tr &    H\tr
    \end{bmatrix}$, and as a consequence,
\begin{align*}
\widetilde\Sigma := \lim_{t \to \infty} \frac{\Expect\left[Y(t)\, Y(t)\tr\right]}{t} = \begin{bmatrix}
    G \\
    H
    \end{bmatrix}\, \left(\lim_{t \to \infty} \frac{\Expect\left[X(t)\, X(t)\tr\right]}{t}\right) \, \begin{bmatrix}
    G\tr &    H\tr
        \end{bmatrix}
\end{align*}
which is equivalent to equation (42) of \cite{kim2011modeling}. Note that, from the above properties of the innovations, 
\begin{align*}
\lim_{t \to \infty} \frac{\Expect\left[X(t)\, X(t)\tr \right]}{t} &= \begin{bmatrix}
															\lim_{t \to \infty} \frac{\Expect\left[ \hl{E}^{(0)}(t)\, \hl{E}^{(0)}(t)\tr \right]}{t} & 0\\
															   0 & \lim_{t \to \infty} \frac{\Expect\left[ \zeta(t)\, \zeta(t)\tr \right]}{t}
                                                            \end{bmatrix}\\
                                                          &= \begin{bmatrix}
															\mbox{diag}(v^2) & 0\\
															   0 & \lim_{t \to \infty} \frac{\Expect\left[ \zeta(t)\, \zeta(t)\tr  \right]}{t}
                                                            \end{bmatrix}\\
                                                           &= \Sigma^{(P)},
\end{align*}
where $\Sigma^{(P)}$ is defined by \eqref{eqn:Sigma_P}, and the last equality follows from the above properties of the innovation processes and the fact that $\lim_{t \to \infty} \frac{\Expect[\hl{A}_i(t)]}{t} = \lambda_i$ for all $i$. Therefore, $$\widetilde\Sigma = \begin{bmatrix} G \\ H \end{bmatrix} \Sigma^{(P)} \begin{bmatrix} G\tr & H\tr \end{bmatrix} = \begin{bmatrix} G\, \Sigma^{(P)}\, G\tr & G\, \Sigma^{(P)}\, H\tr \\ H\, \Sigma^{(P)}\, G\tr & H\, \Sigma^{(P)}\, H\tr \end{bmatrix},$$ and thus
$\sigma_{k_1,k_2} = \widetilde\Sigma_{k_1,k_2}$ and $\sigma_{i_1 \to j_1, i_2 \to j_2} = \widetilde\Sigma_{i_1*K + j_1, i_2*K + j_2}$. 
Observe that the innovations method and the results of Theorem~\ref{thm:main} (i) provide essentially the same expressions.

Finally, observe that the virtue of the innovations method in \cite{kim2011modeling} is that it also allows (and  focuses on) cases where $\rho_i$'s (in that paper) are not 0. However, in such cases, all calculations are merely a heuristic. Further, as that was not the focus of \cite{kim2011modeling}, the limits and model assumptions in that paper are not rigorously justified as in the current paper.

\subsection*{Simulation Results}

To further illustrate our result and explore the effect of different policies and constraints on the variance of flows, we carried out a Monte-Carlo simulation of the example network.

In the simulation we set the service distributions of queue $k$ to be distributed as a sum of two i.i.d. exponential random variables, each with mean $(2 \mu_k)^{-1}$. This results in a so-called Erlang~$2$ distribution (having a squared coefficient of variation of $1/2$) with mean~$\mu_k^{-1}$.  

The arrival process, $\hl{E}_1(\cdot)$, is the more variable of the two arrival processes. It is taken to be a renewal process of inter-arrival times that are distributed as a mixture of two independent exponential random variables (hyper-exponential): \hl{with probability} $1/3$ a mean $2$ exponential and \hl{with probability} $2/3$ a mean $1/2$ exponential. This distribution has mean $1$ and squared coefficient of variation $2$ agreeing with $\alpha_1$ and $v_1^2$ as specified in~\eqref{eq:exampleParams}.

The arrival process, $\hl{E}_2(\cdot)$, is less variable. It is taken to be a renewal process with inter-arrival times that are Erlang $2$ distributed this time with mean $1/4$. This is in agreement with $\alpha_2$ and $v_2^2$ as specified in \eqref{eq:exampleParams}.

\vspace{5pt}

We consider two settings:

\begin{description}
\item {\bf Single-class}: Each queue has a dedicated (separate) server. This is a generalized Jackson network.
\item {\bf Multi-class}: Queues $1$ and $2$ are served by the same server under a non-pre-emptive priority policy giving priority to queue $1$. All other queues have their own server.  Note that in this case the load on the server of queues $1$ and $2$ is $\lambda_1/\mu_1 + \lambda_2/\mu_2 \approx 0.97<1$. That is, it is quite heavily loaded but is still stable. Note in general having a load of less than unity does not immediately imply that the system is stable yet for this simple case it can be shown that stability holds  under such a priority policy (c.f. \cite{bramsonBook2008}). 
\end{description}

Besides exemplifying the correctness of our theoretical results, the goal in this simulation set-up is to illustrate that while the asymptotic variability parameters do not depend on service times and scheduling policies, the shape of the variance curve is in general influenced by such factors. 

\begin{figure}[h!]
\centering
{\includegraphics[width=5.2in]{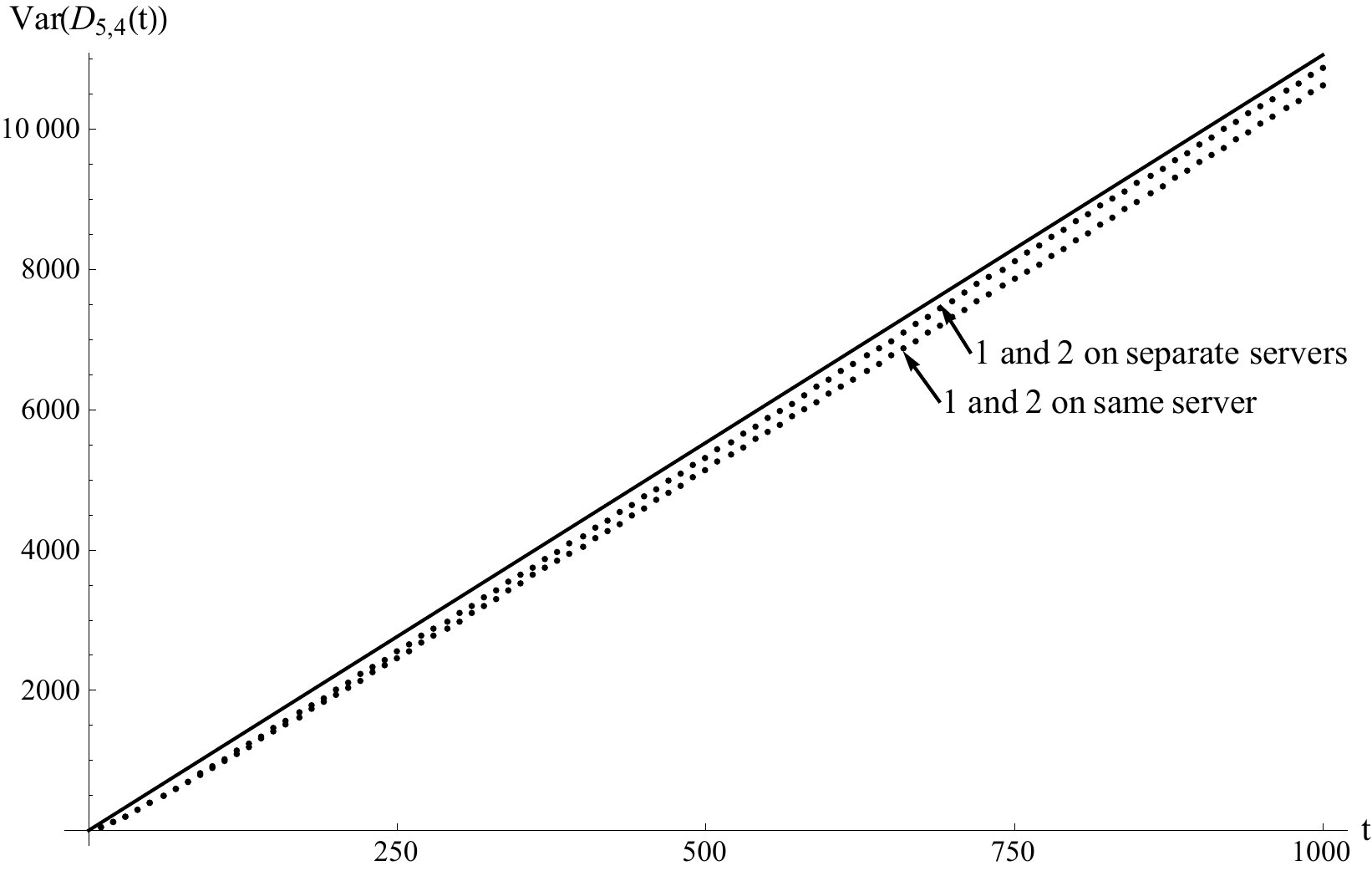}}
{\includegraphics[width=5.6in]{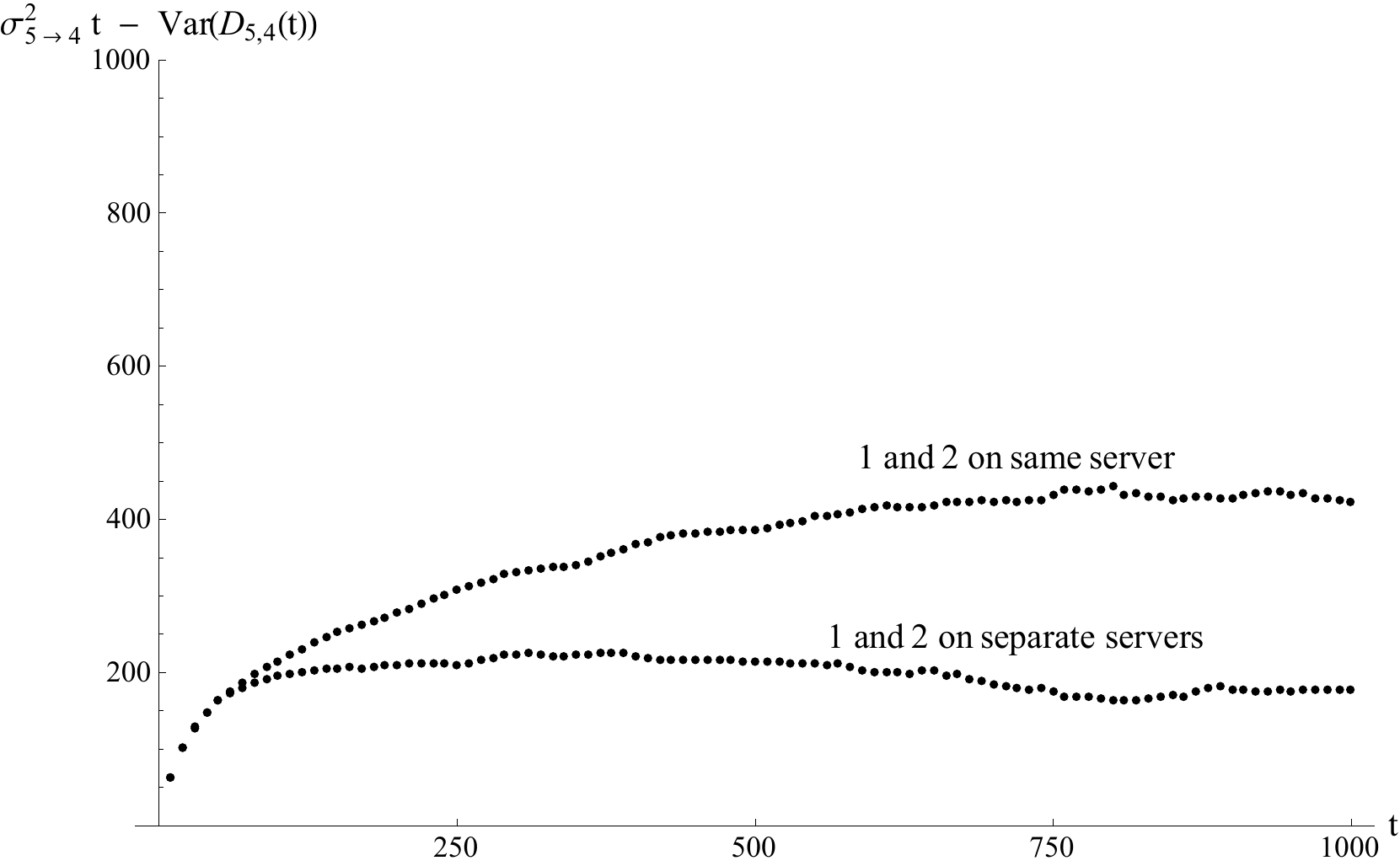}}
\caption{\label{fig2} Simulation estimates of ${\Var\big(D_{5 \to 4}(t)\big)}$ for two cases: single-class (${1}$ and ${2}$ on separate servers) and multi-class (${1}$ and ${2}$ on same server with a priority policy). The top graph illustrates the variance curve estimates (dotted) vs. the solid line ${\sigma^2_{5 \to 4} t}$.  The bottom graph shows the bias:  ${\sigma^2_{5 \to 4} t -\Var\big(D_{5 \to 4}(t)\big)}$. As is illustrated, both systems have the same asymptotic variance for ${D_{5\to4}(t)}$, yet their variance curves differ for finite~${t}$.}
\end{figure}

We ran $2\times 10^5$ simulation runs of each case (single-class and multi-class) each for $1,000$ time units, starting at time $t=0$ with the system empty\footnote{The simulation was carried out using a simulation package written in C++: PRONETSIM. See \cite{nazarathy2008control}, Appendix~A, for details about this software.}. We then estimated $\Var\big(D_{5 \to 4}(t)\big)$ for each run over a grid of time points $t=20,40,60,\ldots,1000$, by taking the sample variance at each time point over $2 \times 10^5$ observations. Note that we purposely observe the flow $5 \to 4$ which is not directly adjacent to the multi-class server serving $1$ and $2$. 

Our main theorem applied to this example implies that in both the single-class and multi-class case, for non-small $t$,
\[
\Var\big(D_{5 \to 4}(t)\big) \approx \sigma^2_{5 \to 4} t = \frac{199}{18} t =  11.05\overline{5} \, t.
\]
This is illustrated in Figure~\ref{fig2} (top) where we plot the variance curves versus the approximation $\sigma^2_{5 \to 4} t$.  To take a closer look at the effect of single-class vs. multi-class we then plot the bias, $\sigma^2_{5 \to 4} t -\Var\big(D_{5 \to 4}(t)\big)$ on Figure~\ref{fig2} (bottom). It is indeed evident that different system characteristics yield different variance curves. 

It is somewhat expected that the multi-class case will have a higher bias, since in this case the server of $1$ and $2$ is under a heavier load ($0.97$). Further, in that case one can expect more ``bursts'' on the flow $2 \to 4$ since queue $2$ is served with low-priority.  These bursts perhaps ``propagate'' to flow $4\to 5$ and ultimately to the flow which we measure: $5 \to 4$. Nevertheless, such phenomena are not captured by the asymptotic quantities found in the current paper. It should be noted that in \cite{hautphenne2015intercept} second order properties of this sort are explored for elementary queueing systems such as the stable M/G/1 queue. It is not clear how to extend such an investigation to networks.

\section{Conclusion}
\label{sec:coclusion}

Prior to this work a rigorous analysis dealing with exact expressions for the asymptotic variability of flows was lacking in the literature.  In this paper we put forward easy computable expressions together with a simple diffusion limit theorem for the flows. 

The queueing networks we considered in this paper are assumed to be open and stable. This stands in contrast with the more general case handled in \cite{chen1991stochastic} (where nodes are allowed to be either under-loaded, over-loaded or critical).  It should be mentioned that our results easily carry over to the case where some nodes are over-loaded. In this case, the service times of over-loaded nodes contribute to the exogenous arrivals in a straightforward manner (see for example \cite{goodman1984nej} for an early treatment of this idea). On the contrary, the case in which some nodes are critical is more challenging. In that case, the single-server queue was only recently handled with some difficulty in \cite{al2010asymptotic}. There the authors observed a BRAVO effect (Balancing Reduces Asymptotic Variance of Outputs). We do not handle this in the network context. Thus the challenge of finding the asymptotic variability of flows in critical queueing networks remains.

\vspace{10pt}


{\bf Acknowledgments}: 
YN is supported by Australian Research Council (ARC) grant DP180101602. Part of the work was carried out while WS was supported by an Ethel Raybould Visiting Fellowship to the University of Queensland. SBM is supported by the Australian Research Council  Centre  of  Excellence  for  Mathematical  and  Statistical Frontiers  (ACEMS)  under  grant  number  CE140100049. Last but not least, we thank an anonymous referee of an earlier version of the paper for drawing our attention to the usefulness of the innovations method calculations.

\bibliography{PaperDatabase,BookDatabase} 

\end{document}